\documentclass[11pt]{article}
\usepackage{amsmath,amssymb}
\usepackage{bm,amsthm}
\usepackage{graphicx,color}
\usepackage{array,cite,url}
\usepackage{amsfonts}
\usepackage{graphicx}
\usepackage{epstopdf}
\usepackage{algorithm}
\usepackage{array}
\usepackage{algcompatible}
\usepackage[shortlabels]{enumitem}
\usepackage{amsopn}
\usepackage{hyperref}
\usepackage{cleveref}

\newcommand{\rar}{\rightarrow}

\newcommand{\h}{{\mathcal{H}}}
\newcommand{\s}{{\sigma}}
\newcommand{\x}{{\bf x}}
\def\f{{\bf f}}
\def\S{{\bf S}}
\newcommand{\z}{{\bf z}}
\newcommand{\T}{{\bf T}}
\def\v{{\bf v}}
\newcommand*{\horzbar}{\rule[.5ex]{2.5ex}{0.5pt}}
\newcommand*{\vertbar}{\rule[-1ex]{0.5pt}{2.5ex}}

\DeclareMathOperator{\spn}{span}

\newtheorem{remark}{Remark}
\newtheorem{lemma}{Lemma}
\newtheorem{proposition}{Proposition}
\newtheorem{corollary}{Corollary}

\title{\LARGE \bf
	Ergodic theory, Dynamic Mode Decomposition and Computation of Spectral Properties of the Koopman operator}

\author{Hassan Arbabi$^{1}$ and Igor Mezi\'c$^{2}$
	\thanks{$^{1}$Hassan Arbabi is with the department of Mechanical Engineering,
		University of California, Santa Barbara, CA, 93106, USA  {\tt\small harbabi@engr.ucsb.edu}}%
	\thanks{$^{2}$Igor Mezi\'c is with the faculty of Mechanical Engineering and Mathematics,  University of California, Santa Barbara, CA, 93106, USA  {\tt\small mezic@engr.ucsb.edu}}%
}

\hypersetup{
  colorlinks   = true, 
  urlcolor     = blue, 
  linkcolor    = blue, 
  citecolor   = red 
}

\begin{document}

\maketitle

\begin{abstract}
We establish the convergence of a class of numerical algorithms, known as Dynamic Mode Decomposition (DMD), for computation of the eigenvalues and eigenfunctions of the infinite-dimensional Koopman operator.  The algorithms act on data coming from observables on a state space, arranged in Hankel-type matrices. The proofs utilize the assumption that  the underlying dynamical system  is ergodic. This includes the classical measure-preserving systems, as well as systems whose attractors support a physical measure.
Our approach relies on the observation that vector projections in DMD can be used to approximate the function projections by the virtue of Birkhoff's ergodic theorem. Using this fact, we show that applying DMD to Hankel data matrices in the limit of infinite-time observations yields the true Koopman eigenfunctions and eigenvalues. We also show that the Singular Value Decomposition, which is the central part of most DMD algorithms,  converges to the Proper Orthogonal Decomposition of observables. We use this result to obtain a representation of the dynamics of systems with continuous spectrum based on the lifting of the coordinates to the space of observables. The numerical application of these methods is demonstrated using well-known dynamical systems and examples from computational fluid dynamics.
\end{abstract}

\\


\section{Introduction}

The Koopman operator theory is an alternative formulation of dynamical systems theory which provides a versatile framework for data-driven study of high-dimensional nonlinear systems. The theory is originated in 1930's through the work of Bernard Koopman and John Von Neumann \cite{Koopman:1931,KoopmanandvonNeumann:1932}.  In particular, Koopman realized that the evolution of observables on the state space of a Hamiltonian system  can be described via a linear transformation, which was later named the Koopman operator.
Years later, the work in \cite{mezic2004comparison} and \cite{Mezic:2005} revived the interest in this formalism by proving the Koopman spectral decomposition, and introducing the idea of Koopman modes. 
This theoretical progress was later complemented by data-driven algorithms for approximation of the Koopman operator spectrum and modes, which has led to a new pathway for data-driven study of high-dimensional systems
 (see e.g. \cite{rowley2009,Schmid:2010,susuki2011nonlinear,GeorgescuandMezic:2015,brunton2016extracting}).
We introduce some basics of the Koopman operator theory for continuous- and discrete-time dynamical systems in this section. The reader is referred to \cite{budisic2012applied} for a more detailed review of theory and application.

Consider a continuous-time dynamical system given by
\begin{equation}
\dot \x=\mathbf{F}(\x),
\label{DSGen}
\end{equation}
on the state space $M$, where $\x$ is a coordinate vector of the state, and $\mathbf{F}$ is a nonlinear vector-valued smooth function, of the same dimension as its argument $\x$. Let $\S^t(\x_0)$ denote the position at time $t$ of trajectory of   \cref{DSGen} that starts at $t=0$ at the state $\x_0$. We  call $\S^t(\x_0)$ the flow generated by \cref{DSGen}.

Denote by $\f$ an arbitrary, vector-valued function from $M$ to $\mathbb{C}^k$. We call $\f$ an \emph{observable} of the system in \cref{DSGen}. The value of $\f$ that is observed on a trajectory starting from $\x_0$ at time $0$ changes with time according to the flow, i.e.,
\begin{equation}
\f(t,\x_0)=\f(\S^t(\x_0)).
\end{equation}
The space of all observables such as $\f$ is a linear vector space, and we can define a family of linear operators $U^t$ with $t\in [0,\infty)$, acting on this vector space, by
\begin{equation}
U^t\f(\x_0)=\f(\S^t(\x_0)).
\label{Koopdef}
\end{equation}
Thus, for a fixed $t$, $U^t$ maps the vector-valued observable $\f(\x_0)$ to $\f(t,\x_0)$. We will call the family of operators $U^t$, indexed by time $t$, the Koopman operator of the continuous-time system (\ref{DSGen}).
In operator theory, such operators defined for general dynamical systems,  are often called composition operators\index{Composition operators}, since $U^t$ acts on observables by composing them with the flow $\S^t$ \cite{sm:1993}.

In discrete-time the definition is even simpler, if
\begin{equation}
\z'=\T(\z),
\label{discDSGen}
\end{equation}
is a discrete-time dynamical system with $\z\in M$ and $\T:M\rar M$, then the associated Koopman operator $U$ is defined by
\begin{align}
U\f(\z)=\f\circ \T(\z). \label{eq:KoopmanDef}
\end{align}
The operator $U$ is linear, i.e.,
\begin{equation}
U(c_1\f_1(\z)+c_2\f_2(\z))=c_1\f_1(\T(\z))+c_2\f_2(\T(\z))= c_1U\f_1(\z)+c_2U\f_2(\z).
\end{equation}
A similar calculation shows the linearity of the Koopman operator for continuous-time systems as well.
We call $\phi:M\rightarrow \mathbb{C}$ an eigenfunction of the Koopman operator $U$, associated with eigenvalue $\lambda\in\mathbb{C}$, when
\begin{align*}
U\phi = \lambda \phi.
\end{align*}
For the continuous-time system, the definition is slightly different:
\begin{align*}
U^t\phi = e^{\lambda t} \phi.
\end{align*}
The eigenfunctions and eigenvalues of the Koopman operator encode lots of information about the underlying dynamical system.
For example, the level sets of certain eigenfunctions determine the invariant manifolds \cite{Mezic:2015}, and the global stability of equilibria can be characterized by considering the eigenvalues and eigenfunctions of the Koopman operator \cite{mauroy2014global}. Another outcome of the this theory, which is specially useful for high-dimensional systems, is the Koopman mode decomposition (KMD). If the Koopman spectrum consists of only eigenvalues (i.e. discrete spectrum), the evolution of observables can be expanded in terms of the Koopman eigenfunctions, denoted by $\phi_j,~j=0,1,\ldots$, and Koopman eigenvalues $\lambda_j$. Consider again the observable $\f:M\rightarrow \mathbb{C}^k$. The evolution of $\f$ under the discrete system in \cref{discDSGen} is given by
\begin{align}
U^n \f(\z_0) := \f\circ \T^n(\z_0) = \sum_{j=1}^{\infty} \mathbf{v}_j \phi_j(\z_0) \lambda_j^n \label{eq:KMD}.
\end{align}
In the above decomposition, $\mathbf{v}_j\in \mathbb{C}^k$ is the Koopman mode associated with the pair $(\lambda_j,\phi_j)$ and  is given by the projection of the observable $\f$ onto the eigenfunction $\phi_j$. These modes correspond to components of the physical field characterized by exponential growth and/or oscillation in time and play an important role in the analysis of large systems (see the references in the first paragraph).

In the recent years, a variety of methods have been developed for computation of the Koopman spectral properties (eigenvalues, eigenfunctions and modes) from data sets that describe the evolution of observables such as $\f$. A large fraction of these methods belong to the class of algorithms known as Dynamic Mode Decomposition, or DMD in short (see e.g. \cite{rowley2009,Schmid:2010,tu2014dynamic,williams2015data}).
In this paper, we prove that the eigenvalues and eigenfunctions obtained by a class of DMD algorithms converge to the  eigenvalues and eigenfunctions of the Koopman operator for ergodic systems. Such proofs - that finite-dimensional approximations of spectra converge to  spectra of infinite-dimensional linear operators - are still rare and mostly done for self-adjoint operators \cite{hansen:2010} (Koopman operator is typically not self-adjoint).
Our approach here provides a new - ergodic theory inspired - proof, the strategy of which could be used in other contexts of non-self adjoint operators.

Our methodology for computation of the Koopman spectrum is to apply DMD to Hankel matrix of data.
The Hankel matrix is created by the delay-embedding of time series measurements on the observables.
Delay-embedding is an established method for geometric reconstruction of attractors for nonlinear systems based on measurements of generic observables \cite{Takens:1981,sauer1991embedology}. By combining delay-embedding with DMD, we are able to extract analytic information about the state space (such as the frequency of motion along the quasi-periodic attractor or the structure of isochrons) which cannot be computed from geometric reconstruction.
On the other hand, Hankel matrices are extensively used in the context of linear system identification (e.g. \cite{juang1985eigensystem,van2012subspace}).
The relationship between KMD and linear system identification methods was first pointed out in \cite{tu2014dynamic}. It was shown that applying DMD to the Hankel data matrix recovers the same linear system, up to a similarity transformation, as the one obtained by Eigensystem Realization Algorithm (ERA)\cite{juang1985eigensystem}. In a more recent study, Brunton and co-workers \cite{brunton2016chaos} proposed a new framework for Koopman analysis using the Hankel-matrix representation of data. Using this framework, they were able to extract a linear system with intermittent forcing that could be used for local predictions of chaotic systems (also see \cite{brunton2016koopman} and \cite{brunton2016extracting}).
Our work strengthens the above results by providing a rigorous connection between linear analysis of delay-embedded data and identification of nonlinear systems using the Koopman operator theory.
We establish the convergence of our algorithm - called \textit{Hankel-DMD} -  based on the observation that for ergodic systems the projections of data vectors converge to function projections in the space of observables. This observation has already been utilized in a different approach for computation of the Koopman spectral properties \cite{giannakis2015data}. We also note that the Hankel-DMD algorithm is closely related to the Prony approximation of KMD \cite{susuki2015prony}, and it can interpreted as a variation of Extended Dynamic Mode Decomposition \cite{williams2015data} on ergodic trajectories (also see \cite{klus2015numerical}).

The outline of this paper is as follows:
In  \cref{sec:IntroDMD}, we describe the three earliest variants of DMD, namely, the companion-matrix DMD, SVD-enhanced DMD and Exact DMD.
In  \cref{sec:ErgodicSystems}, we review some elementary ergodic theory, the Hankel representation of data and prove the convergence of the companion-matrix Hankel-DMD method.
In  \cref{sec:VanderPol}, we extend the application of this method to observations on trajectories that converge to an ergodic attractor.
In  \cref{sec:SVD-POD}, we point out a new connection between the Singular Value Decomposition (SVD) on data matrices and Proper Orthogonal Decomposition (POD) on the ensemble of observables on ergodic dynamical systems. By using this interpretation of SVD, we are able to show the convergence of the Exact DMD for ergodic systems in \cref{sec:ExactDMD}. These results enable us to extract the Koopman spectral properties from measurements on multiple observables. We recapitualte the Hankel-DMD algorithm and present some numerical examples in \cref{sec:applications}.
 We summarize our results in \cref{sec:conclusion}.

\section{Review of Dynamic Mode Decomposition (DMD)}
\label{sec:IntroDMD}
DMD was originally introduced as a data analysis technique for complex fluid flows by P.J. Schmid \cite{schmid2008,Schmid:2010}. The primary goal of this algorithm was to extract the spatial flow structures that evolve linearly with time, i.e., the structures that grow or decay exponentially - possibly with complex exponents. The connection between this numerical algorithm and the linear expansion of observables in KMD was first noted in \cite{rowley2009}, where a variant of this algorithm was used to compute the Koopman modes of a jet in cross flow.
The initial success of this algorithm in the context of fluid mechanics motivated an ongoing line of research on data-driven analysis of high-dimensional and complex systems using the Koopman operator theory, and consequently, a large number of DMD-type algorithms have been proposed in the recent years for computation of the Koopman spectral properties \cite{brunton2013compressive,tu2014dynamic,williams2015data,proctor2016dynamic,kutz2016multiresolution}.

The three variants of DMD that we consider in this work are  the \textit{companion-matrix DMD}  \cite{rowley2009}, the \textit{SVD-enhanced DMD} \cite{Schmid:2010}, and the \textit{Exact DMD} \cite{tu2014dynamic}.
The SVD-enhanced and Exact variants of DMD are more suitable for numerical implementation, while the companion-matrix method enables a  more straightforward proof for the first result of this paper.
In the following, we first describe the mathematical settings for application of DMD, and then describe the above three algorithms and the connections between them. Let
\begin{align}
\mathbf{f}:=\left[
\begin{matrix}
f_1 \\ f_2 \\ \vdots \\f_n
\end{matrix}
\right]:M\rar \mathbb{R}^n
\end{align}
be a vector-valued observable defined on the dynamical system in \cref{discDSGen}, and let
\begin{align}
D:=\left[
\begin{matrix}
f_1(z_0) & f_1\circ T(z_0) & \ldots & f_1\circ T^{m}(z_0)  \\
f_2(z_0) & f_2\circ T(z_0) & \ldots & f_2\circ T^{m}(z_0)  \\
\vdots & \vdots & \ddots & \vdots  \\
f_n(z_0) & f_n\circ T(z_0) & \ldots & f_n\circ T^{m}(z_0)
\end{matrix}
\right] \label{eq:DataMatrix}
\end{align}
be the matrix of measurements recorded on $\f$ along a trajectory starting at the initial condition $z_0\in M$. Each column of $D$ is called a \emph{data snapshot} since it contains the measurements on the system at a single time instant.
Assuming only discrete eigenvalues for the Koopman operator, we can rewrite the Koopman mode expansion in \cref{eq:KMD} for each snapshot in the form of
  \begin{align}
    D_i := \left[
    \begin{matrix}
  f_1\circ T^{i}(z_0)  \\
  f_2\circ T^{i}(z_0)  \\
  \vdots  \\
  f_n\circ T^{i}(z_0)
    \end{matrix}
    \right] =  \sum_{j=1}^\infty \lambda_j^i {\mathbf{v}}_j
    \label{eq:KMD2}
  \end{align}
  by absorbing the scalar values of $\phi_j(z_0)$ into the mode $\v_j$.
 In numerical approximation of the Koopman modes, however, we often assume this expansion is finite dimensional and use
   \begin{align}
    \mathbf{f}^i = \sum_{j=1}^n \tilde\lambda_j^i \tilde{\mathbf{v}}_j \label{eq:KMD3}
  \end{align}
 where $\tilde{\mathbf{v}}_j$ and $\tilde\lambda_j$ are approximations to the Koopman modes and eigenvalues in \cref{eq:KMD2}. This expansion resembles the spectral expansion for a linear operator acting on $\mathbb{R}^n$. This operator, which maps each column of $D$ to the next is called the \emph{DMD operator}.
 The general strategy of DMD algorithms is to construct the DMD operator, in the form of a matrix, and extract the dynamic modes and eigenvalues from the spectrum of that matrix.
 In the companion-matrix algorithm (\cref{alg:CDMD}), as the name suggests, the DMD operator is realized in the form of a companion matrix:
\begin{algorithm}[H]
\caption{Companion-matrix DMD}
\label{alg:CDMD}
\begin{algorithmic}[1]
\STATEx{Consider the data matrix $D$ defined in \cref{eq:DataMatrix}.}
\STATE{Define $X=[D_0~D_1~\ldots~D_{m-1}]$}
\STATE{Form the companion matrix
 	 \begin{align}
 	 \tilde C=\left(
 	 \begin{matrix}
 	 0 &  0  & \ldots & 0 &\tilde c_{0}\\
 	 1 &  0 & \ldots & 0& \tilde c_{1}\\
 	 0 &  1 & \ldots & 0& c_{2}\\
 	 \vdots & \vdots & \ddots &  \vdots & \vdots\\
 	 0  &   0       &\ldots & 1 & c_{m-1}
 	 \end{matrix}\right),
 	 \label{eq:compDMD}
 	 \end{align}
with
	\begin{align*}
 	 \left(
 	 \begin{matrix}
 	 c_{0},~
 	 c_{1},
 	 c_{2},~
 	 \ldots ,~
 	 c_{m-2}
 	 \end{matrix}\right)^T = X^\dagger D_{m}.
	\end{align*}
	The $X^\dagger$ denotes the Moore-Penrose pseudo-inverse of $X$ \cite{trefethen1997numerical}.}
\STATE{Let $(\lambda_j,w_j),~j=1,2,\ldots,m$ be the eigenvalue-eigenvector pairs for $\tilde C$. Then $\lambda_j$'s are the dynamic eigenvalues. Dynamic modes $\tilde v_j$ are given by
	\begin{align}
	\tilde v_j=Xw_j,\quad j=1,2,\ldots,m.
	\end{align} }
\end{algorithmic}
\end{algorithm}

In the above algorithm, the companion matrix $\tilde C$ is the realization of the DMD operator in the basis which consists of the columns in $X$. The pseudo-inverse in step 2 is used to project the last snapshot of $D$ onto this basis. In case that $D_m$ lies in the range of $X$, we have
\begin{align}
r:= D_m - X(X^\dagger D_m)=0
\end{align}
which means that the companion matrix $\tilde C$ exactly maps each column of $D$ to the next. If columns of $X$ are linearly dependent, however, the above  projection is not unique, and the problem of determining the DMD operator is generally over-constrained. Furthermore, when $X$ is ill-conditioned, the projection in step 2 becomes numerically unstable.

The SVD-enhanced DMD algorithm (\cref{alg:SVDDMD}), offers a more robust algorithm for computation of dynamic modes and eigenvalues:
\begin{algorithm}[H]
\caption{SVD-enhanced DMD}
\label{alg:SVDDMD}
\begin{algorithmic}[1]
\STATEx{Consider the data matrix $D$ defined in \cref{eq:DataMatrix}.}
\STATE{Define $X=[D_0~D_1~\ldots~D_{m-1}]$ and $Y=[D_1~D_2~\ldots~D_{m}]$}
\STATE{Compute the SVD of $X$:
 		\begin{align*}
 		X=WS\tilde V^*.
 		\end{align*}
 		\item Form the matrix
 		\begin{align*}
 		\hat A=W^*Y\tilde V S^{-1}.
 		\end{align*}}
\STATE{Let $(\lambda_j,w_j),~j=1,2,\ldots,m$ be the eigenvalue-eigenvector pairs for $\hat A$. Then $\lambda_j$'s are the dynamic eigenvalues. Dynamic modes $\tilde v_j$ are given by
 		\begin{align*}
 			\tilde v_j=Ww_j,\quad j=1,2,\ldots,m.
 		\end{align*} }
\end{algorithmic}
\end{algorithm}

In this method, the left singular vectors $W$ are used as the basis to compute a realization of the DMD operator, which is $\hat A$. In fact, column vectors in $W$ form an orthogonal basis which enhances the numerical stability of the projection process (the term $W^*Y$ in step 3).
When $X$ is full rank and $\lambda_j$'s are distinct, the dynamic modes and eigenvalues computed by this algorithm is the same as the companion-matrix algorithm \cite{chen2012variants}.

The Exact DMD algorithm (\cref{alg:EDMD}) generalizes the SVD-enhanced algorithm to the case where the sampling of the data might be non-sequential. For example, consider the data matrices
\begin{align}
X=\left[
\begin{matrix}
f_{1}(z_0) & f_{1}(z_1) & \ldots & f_{1}(z_m) \\
f_{2}(z_0) & f_{2}(z_1) & \ldots & f_{2}(z_m)  \\
\vdots & \vdots & \ddots & \vdots  \\
f_{n}(z_0) & f_{n}(z_1) & \ldots & f_{n}(z_m)
\end{matrix}
\right] \label{eq:XDataMatrix}
\end{align}
and
\begin{align}
Y=\left[
\begin{matrix}
f_{1}\circ T(z_0) & f_{1}\circ T(z_1) & \ldots & f_{1}\circ T(z_m) \\
f_{2}\circ T(z_0) & f_{2}\circ T(z_1) & \ldots & f_{2}\circ T(z_m)  \\
\vdots & \vdots & \ddots & \vdots  \\
f_{n}\circ T(z_0) & f_{n}\circ T(z_1) & \ldots & f_{n}\circ T(z_m)
\end{matrix},
\right] \label{eq:YDataMatrix}
\end{align}
where $\{z_0,z_1,\ldots,z_m\}$ denotes a set of arbitrary states of the dynamical system in \cref{discDSGen}.
The Exact DMD algorithm computes the operator that maps each column of $X$  to the corresponding column in $Y$.

\begin{algorithm}
\caption{Exact DMD}
\label{alg:EDMD}
\begin{algorithmic}[1]
\STATEx{Consider the data matrices $X$ and $Y$ defined in \cref{eq:XDataMatrix} and \cref{eq:YDataMatrix}.}
\STATE{Define $X=[D_0~D_1~\ldots~D_{m-1}]$ and $Y=[D_1~D_2~\ldots~D_{m}]$}
\STATE{Compute the SVD of $X$:
 		\begin{align*}
 		X=WS\tilde V^*.
 		\end{align*}}
\STATE{ Form the matrix
 		\begin{align*}
 		\hat A=W^*Y\tilde V S^{-1}.
 		\end{align*}}
\STATE{Let $(\lambda_j,w_j),~j=1,2,\ldots,m$ be the eigenvalue-eigenvector pairs for $\hat A$. Then $\lambda_j$'s are the dynamic eigenvalues. }
\STATE{	The exact dynamic modes $\tilde v_j$ are given by
 		\begin{align*}
 		\tilde v_j=\frac{1}{\lambda_j} Y\tilde V S^{-1} w_j,\quad j=1,2,\ldots,m.
 		\end{align*}}
\STATE{ The projected dynamic modes $\chi_j$ are given by
 		 \begin{align*}
 		  		\chi_j=W w_j,\quad j=1,2,\ldots,m.
 		 \end{align*}}
\end{algorithmic}
\end{algorithm}

The finite-dimensional operator that maps the columns of $X$ to $Y$ is known as the Exact DMD operator, with the explicit realization,
\begin{align}
\tilde A=YX^\dagger.
\end{align}
The matrix $\tilde A$ is not actually formed in \cref{alg:EDMD}, however, the dynamic eigenvalues and exact dynamic modes form the eigen-decomposition of $\tilde A$. We note that the projected dynamic modes and exact dynamic modes coincide if the column space of $Y$ lie in the range of $X$. Moreover, applying Exact DMD to $X$ and $Y$ matrices defined in algorithm \ref{alg:SVDDMD} yields the same eigenvalues and modes as SVD-enhanced DMD.

In the following, we will show how DMD operators converge to a finite-dimensional representation of the Koopman operator for ergodic systems. The critical observation that enables us  to do so, is the fact that vector projections in the DMD algorithm can be used to approximate the projections in the function space of observables.

\section{Ergodic theory and Hankel-matrix representation of data}
\label{sec:ErgodicSystems}
In this section, we recall the elementary ergodic theory and give a new interpretation of Hankel-matrix representation of data in the context of the Koopman operator theory. The main result of this section is \cref{prop:Hankel-DMD} which asserts the convergence of companion-matrix Hankel-DMD for computation of Koopman spectrum.
Despite the intuitive proof of its convergence, this method is not well-suited for numerical practice, and a more suitable alternative for numerical computation will be presented in \cref{sec:ExactDMD,sec:applications}. In \cref{sec:VanderPol}, we present analogous results for the basin of attraction of ergodic attractors.

Consider the dynamics on a compact invariant set $A$, possibly the attractor of a dissipative dynamical system, given by the measure-preserving map $T:A\rightarrow A$. Let $\mu$ be the preserved measure with $\mu(A)=1$, and assume that for every invariant set $B\subset A$, $\mu(B)=0$ or $\mu(A-B)=0$, i.e., the map $T$ is ergodic on $A$.
A few examples of ergodic sets in dynamical systems are limit cycles, tori with uniform flow and chaotic sets like Lorenz attractor.
We define the Hilbert space $\h$ to be the set of all observables on $A$ which are square-integrable with respect to the measure $\mu$, i.e.,
\begin{equation}
  \h:=\{f:A\rar \mathbb{R}~s.t.~\int_A |f|^2 d\mu<\infty\}. \label{eq:HilbertDef}
\end{equation}
 The Birkhoff's ergodic theorem \cite{Petersen:1983} asserts the existence of infinite-time average of such observables and relates it to the spatial average over the set $A$. More precisely, if $f\in\h$, then
\begin{align}
\lim_{N \rightarrow \infty}\frac{1}{N}\sum_{k=0}^{N-1} f\circ T^k(z)=\int_A f d\mu,~~\text{for almost every } z\in A. \label{eq:ergodicavg}
\end{align}
An important consequence of this theorem is that the inner products of observables in $\h$ can be approximated using the time series of observations. To see this, denote by $\tilde f_m(z_0)$ and $\tilde g_m(z_0)$ the vector of $m$ sequential observations made on observables $f,g\in\h$ along a trajectory starting at $z_0$,
\begin{subequations}
\begin{align}
\tilde f_m (z_0)= [f(z_0),~f\circ T(z_0),~\ldots,~f\circ T^{m-1}(z_0)] \label{eq:observationVector}\\
\tilde g_m (z_0)= [g(z_0),~g\circ T(z_0),~\ldots,~g\circ T^{m-1}(z_0)]
\end{align}
\end{subequations}
Then for almost every $z_0 \in A$,
\begin{align}
\lim_{m\rightarrow\infty} \frac{1}{m} <\tilde f_m (z_0),\tilde g_m(z_0)> = \lim_{m \rightarrow \infty}\frac{1}{m}\sum_{k=0}^{m-1} (f g^*)\circ T^k(z_0)
= \int_A fg^* d\mu
&=<f,g>_\h    \label{eq:innerproduct}
\end{align}
where we have used $<.,.>$ for vector inner product and $<.,.>_\h$ for the inner product of functions in $\h$. The key observation in this work is that using the data vectors such as $\tilde f_m(z_0)$ we can approximate the projection of observables onto each other according to \cref{eq:innerproduct}.

Now consider the longer sequence of observations
\begin{align}
\nonumber  \tilde{f}_{m+n}=[f(z_0),~f\circ T(z_0),~\ldots,~f\circ T^{m-1}(z_0),~\ldots,f\circ T^{m+n-1}(z_0)]
\end{align}
which could be rearranged into a Hankel matrix by delay-embedding of dimension $m$,
\begin{align}
\tilde H &= \left(
\begin{matrix}
f(z_0) &  f\circ T(z_0)  & \ldots & f\circ T^n(z_0)\\
f\circ T(z_0) &  f\circ T^2(z_0) & \ldots & f\circ T^{n+1}(z_0)\\
\vdots & \vdots & \ddots &   \vdots\\
f\circ T^{m-1}(z_0)  &   f\circ T^{m}(z_0)       &\ldots & f\circ T^{m+n-1}(z_0) \\
\end{matrix}\right)
\label{eq:Hankel}
\end{align}
Given the definition of the Koopman operator in \cref{eq:KoopmanDef}, we also observe that $j-$th column of this matrix is the sampling of the observable $U^{j-1}f$ along the same trajectory, and we can rewrite it in a more compact form,
\begin{align*}
\tilde H = \left(\tilde{f}_m,~U\tilde{f}_m,\ldots,~U^{n}\tilde{f}_m \right).
\end{align*}
This matrix can be viewed as a sampling of the Krylov sequence of observable $f$, defined as 
\begin{align*}
 \mathcal{F}_n:=[f,~Uf,~\ldots,~U^nf].
\end{align*}

The basic idea of the Hankel-DMD method is to extract the Koopman spectra from this sequence, which is analogous to the idea of Krylov subspace methods for computing the eigenvalues of large matrices \cite{saad2011numerical}.
A simplifying assumption that we utilize in most of this paper is that there exists a finite-dimensional subspace of $\h$ which is invariant under the action of the Koopman operator and contains our observable of interest $f$. In general, the existence of finite-dimensional Koopman-invariant subspaces is equivalent to the existence of the eigenvalues (i.e. discrete spectrum) for Koopman operator. To be more precise, if such invariant subspace exists, then the Koopman operator restricted to this subspace can be realized in the form of a finite-dimensional matrix and therefore it must have at least one (complex) eigenvalue. Conversely, if the Koopman operator has eigenvalues, the span of a finite number of associated eigenfunctions forms an invariant subspace. However, it is not guaranteed that the arbitrary observables such as $f$ are contained within such subspaces.

Let $k$ be the dimension of the minimal Koopman-invariant subspace, denoted by $\mathcal{K}$, which contains $f$. Then the first $k$ iterates of $f$ under the action of the Koopman operator span $\mathcal{K}$, i.e.,
\begin{equation}
  \mathcal{K}=\spn(\mathcal{F}_n), \text{ for every}~n\geq k-1
\end{equation}

This condition follows from the fact that the Koopman eigenvalues are simple for ergodic systems \cite{halmos1956lectures}, and as a result, $f$ is cyclic in $\mathcal{K}$ \cite{maccluer2008elementary}.
The following proposition shows that the eigenvalues and eigenfunctions obtained by applying DMD to $\tilde H$ converge to true eigenfunctions and eigenvalues of the Koopman operator.
Our proof strategy is to show that companion matrix formed in \cref{alg:CDMD} approximates the $k-$by$-k$ matrix which represents the Koopman operator restricted to $\mathcal{K}$.

\begin{proposition}[\textbf{Convergence of the companion-matrix Hankel-DMD algorithm}] \label{prop:Hankel-DMD} Let the dynamical system in \cref{discDSGen} be ergodic, and $\mathcal{F}_n=[f,~Uf,\ldots,U^nf]$ span a $k$-dimensional subspace of $\h$ (with $k<n$) which is invariant under the action of the Koopman operator.
	 Consider the dynamic eigenvalues and dynamic modes obtained by applying the companion-matrix DMD   (\cref{alg:CDMD}) to the first $k+1$ columns of the Hankel matrix $\tilde H_{m\times n}$ defined in \cref{eq:Hankel}.

	Then, for almost every $z_0$, as $m\rightarrow \infty$:
\newline (a) The dynamic eigenvalues converge to the Koopman eigenvalues associated with the $k$-dimensional subspace.
\newline (b) The dynamic modes converge to the sampling of associated Koopman eigenfunctions on the trajectory starting at $z_0$.
\end{proposition}

\begin{proof}
Consider the first $k$ elements of $\mathcal{F}_n$,
	\begin{align}
	\left( f,~Uf,~\ldots,U^{k-1}f \right) \label{eq:HilbertBasis}
	\end{align}
which are linearly independent. These observables provide a basis for $\mathcal{K}$, and the restriction of Koopman operator to $\mathcal{K}$ can be (exactly) realized as the companion matrix
	$$C=\left(
	\begin{matrix}
	0 &  0  & \ldots & 0 &c_{0}\\
	1 &  0 & \ldots & 0& c_{1}\\
	0 &  1 & \ldots & 0& c_{2}\\
	\vdots & \vdots & \ddots &  \vdots & \vdots\\
	0  &   0       &\ldots & 1 & c_{k-1} \\
	\end{matrix}\right)
	$$
where the last column is the coordinate vector of the function $U^kf$ in the basis, and it is given by
	\begin{align}
	\begin{bmatrix}
	c_0\\c_1\\\vdots\\c_{k-1}
	\end{bmatrix}
	= G^{-1}   \begin{bmatrix}
	<f,U^kf>_\h\\<Uf,U^kf>_\h\\\vdots\\<U^{k-1}f,U^kf>_\h
	\end{bmatrix} \label{eq:projectHilbert}
	\end{align}
Here, $G$ is the Gramian matrix of the basis given by $$G_{ij}=<U^{i-1}f,U^{j-1}f>_\h.$$
Now consider the numerical companion-matrix DMD algorithm and let $X$ be the matrix that contains the first $k$ columns of $\tilde H$.
When applied to the first $k+1$ columns of $\tilde H$, the algorithm seeks the eigenvalues of the companion matrix $\tilde C$, whose last column is given by
	\begin{align}
	\begin{bmatrix}
	\tilde c_0\\\tilde c_1\\\vdots\\\tilde c_{k-1}
	\end{bmatrix}
	= X^\dagger U^k\tilde f_m
	=\tilde{G}^{-1}
	\begin{bmatrix}
	\frac{1}{m}<\tilde f_m,U^k\tilde f_m>\\ \frac{1}{m}<U\tilde f_m,U^k \tilde f_m>\\\vdots\\ \frac{1}{m} <U^{k-1}\tilde f_m,U^k \tilde f_m> \label{eq:projectDMD}
	\end{bmatrix}
	\end{align}
In the second equality, we have used the following relationship for the Moore-Penrose pseudo-inverse of a full-rank data matrix $X$,
\begin{align*}
    X^\dagger = (X^*X)^{-1}X^* = (\frac{1}{m} X^*X)^{-1}(\frac{1}{m}X^*):= \tilde G^{-1}(\frac{1}{m}X^*).
\end{align*}
and defined the numerical Gramian matrix by $$\tilde G_{ij} = \frac{1}{m}<U^{i-1}\tilde f_m,U^{j-1}\tilde f_m>.$$

The averaged inner products in the rightmost vector of \cref{eq:projectDMD} converge to the vector of Hilbert-space inner products in \cref{eq:projectHilbert}, due to \cref{eq:innerproduct}.
The same argument suggests element-wise convergence of the numerical Gramian matrix to the $G$ in \cref{eq:projectHilbert}, i.e.,
	\begin{align*}
	\lim_{m\rightarrow \infty } \tilde G_{ij} = G_{ij}.
	\end{align*}
Furthermore
	\begin{align*}
	\lim_{m\rightarrow \infty } \tilde G ^{-1} =\left( \lim_{m\rightarrow \infty } \tilde G\right) ^{-1}=G^{-1}
	\end{align*}
We have interchanged the limit and inverting operations in the above since $G$ is invertible (because the basis is linearly independent). Thus the DMD operator $\tilde C$ converges to the Koopman operator realization $C$.
The eigenvalues of matrix are depend continuously on its entries which guarantees the convergence of the eigenvalues of $\tilde C$ to the eigenvalues of $C$ as well. This proves the statement in (a).
Now let $v_k$ be the set of normalized eigenvectors of $C$, that is,
	\begin{align*}
	Cv_j=\lambda_j v_j,~ \|v_j\|=1,~j=1,\ldots,k.
	\end{align*}
	These eigenvectors give the coordinates of Koopman eigenfunctions in the basis of \cref{eq:HilbertBasis}.
	Namely, $\phi_j,~j=1,\ldots,k$ defined by
	\begin{align}
	\phi_i = \left( f,~Uf,~\ldots,U^{k-1}f \right)  v_i,\quad i=1,\ldots,k.
	\end{align}
are a set of Koopman eigenfunctions in the invariant subspace.
Given the convergence of $\tilde C$ to $C$ and convergence of their eigenvalues, the normalized eigenvectors of $\tilde C$, denoted by $\tilde v_j,~j=0,1,\ldots,k$ also converge to $v_j$'s. We define the set of candidate functions by
	\begin{align}
	\tilde \phi_i = \left( f,~Uf,~\ldots,U^{k-1}f \right) \tilde v_i,\quad i=1,\ldots,k.
	\end{align}
and show that they converge to $\phi_i$ as $m\rightarrow\infty$.  Consider an adjoint basis of \cref{eq:HilbertBasis} denoted by $\{g_j\},~j=0,1,..,k-1$ defined such that
	$<g_i,U^jf>_\h=\delta_{ij}$ with $\delta$ being the Kronecker delta.
We have
	\begin{align*}
	\lim_{m\rightarrow \infty} <\tilde \phi_i,g_j>_\h = \lim_{m\rightarrow \infty} \tilde v_{ij} = v_{ij}= <\phi_i,g_j>_\h
	\end{align*}
where $v_{ij}$ is the $i-$th entry of $v_j$. The above statement shows the \emph{weak} convergence of the $\tilde \phi_j$ to $\phi_j$ for $j=1,\ldots,k$. However, both set of functions belong to the same \emph{finite}-dimensional subspace and therefore weak convergence is strong convergence. The $j$-th dynamic mode given by
\begin{align}
w_i=\left( \tilde f_m,~ U\tilde f_m,~\ldots, U^{k-1} \tilde f_m\right)\tilde v_i
\end{align}
is the sampling of $\tilde\phi_j$ along the trajectory, and convergence of $\tilde \phi_j$ means that $w_j$ converges to the value of Koopman eigenfunction $\phi_j$ on the trajectory starting at $z_0$.  The proposition is valid for almost every initial condition for which the ergodic average in \cref{eq:ergodicavg} exists.
\end{proof}

\begin{remark}\label{remark:ergodicsampling}
	In the above results, the data vectors can be replaced with any sampling vector of the observables that satisfy the convergence of inner products as in \cref{eq:innerproduct}. For example, instead of using $\tilde f$ as defined in \cref{eq:observationVector}, we can use the sampling vectors of the form
	\begin{align*}
	\hat f_m (z_0)= [f(z_0),~f\circ T^l(z_0),~f\circ T^{2l}(z_0),~\ldots,~f\circ T^{(m-1)l}(z_0)]
	\end{align*}
	where $l$ is a positive finite integer.
\end{remark}

\subsection{Extension of Hankel-DMD to the basin of attraction} \label{sec:VanderPol}
Consider the ergodic set $A$ to be an attractor of the dynamical system \cref{discDSGen} with a basin of attraction $\mathcal{B}$. { The existence of ergodic average in \cref{eq:ergodicavg} can be extended to trajectories starting in $\mathcal{B}$ by assuming that the invariant measure on $A$ is a physical measure \cite{eckmann1985ergodic,Young:2002}.
To formalize this notion, let $\nu$ denote the standard Lebesgue measure on $\mathcal{B}$. We assume that there is a subset $B\subset \mathcal{B}$ such that $\nu(\mathcal{B}-B)=0$ and for every initial condition in $B$ the ergodic averages of continuous functions exist. That is, if $f:\mathcal{B}\rar \mathbb{R}$ is continuous, then
\begin{align}
\lim_{N \rightarrow \infty}\frac{1}{N}\sum_{k=0}^N f\circ T^k(z)=\int_A f d\mu,~~\text{for $\nu-$almost every } z\in \mathcal{B}. \label{eq:ergodicavg2}
\end{align}
Roughly speaking, this assumption implies that the invariant measure $\mu$ rules the asymptotics of almost every trajectory in $\mathcal{B}$, and therefore it is relevant for physical observations and experiments.
Using \cref{eq:ergodicavg2}, we can extend \cref{prop:Hankel-DMD} to the trajectories starting almost everywhere in $\mathcal{B}$.
The only extra requirement is that  \emph{the observable must be continuous in the basin of attraction}.

\begin{proposition}[\textbf{Convergence of Hankel-DMD inside the basin of attraction}] \label{prop:Hankel-DMD-basin}
Let $A$ be the ergodic attractor of the dynamical system \cref{discDSGen} with the basin of attraction $\mathcal{B}$ which supports a physical measure.
Assume $f:\mathcal{B}\rar \mathbb{R}$ is a continuous function with $f|_A$ belonging to a $k$-dimensional Koopman-invariant subspace of $\h$. Let $\tilde H$ be the Hankel matrix \cref{eq:Hankel} of observations on $f$ along the trajectory starting at $z_0\in\mathcal{B}$ with $n>k$.
Consider the dynamic modes and eigenvalues obtained by applying the companion-matrix DMD (\cref{alg:CDMD}) to the first $k+1$ columns of $\tilde H$.

Then, for $\nu-$almost every $z_0$, as $m\rar\infty$:
\newline (a) The dynamic eigenvalues converge to the Koopman eigenvalues.
\newline (b) The dynamic modes converge to the value of associated eigenfunctions $\phi_j$ along the trajectory starting at $z_0$.
\end{proposition}
}
\begin{proof}
The proof of (a) is similar to \cref{prop:Hankel-DMD} and follows from the extension of ergodic averages to the basin of attraction by \cref{eq:ergodicavg2}.
To show that the dynamic mode $w_j$ converges to Koopman eigenfunctions along the trajectory, we need to consider the evolution of $w_j$ under the action of the Koopman operator:
\begin{align*}
  \lim_{m\rar \infty} U w_j &= \lim_{m\rar \infty} U [\tilde f_m, U\tilde f_m,\ldots, U^{k-1}\tilde f_m]\tilde v_j \\
                            &= \lim_{m\rar \infty}  [U\tilde f_m, U^2\tilde f_m,\ldots, U^{k}\tilde f_m]\tilde v_j \\
                            &= \lim_{m\rar \infty}  [\tilde f_m, U\tilde f_m,\ldots, U^{k-1}\tilde f_m]\tilde C\tilde v_j \\
                            &= \lim_{m\rar \infty}  [U\tilde f_m, U^2\tilde f_m,\ldots, U^{k}\tilde f_m]\tilde \lambda_j\tilde v_j\\
                            &= \lim_{m\rar \infty}  \tilde \lambda_j w_j\\
                            & = \lambda_j w_j.
\end{align*}	
Therefore, $w_j$ converges to the sampling of values of the eigenfunction associated with eigenvalue $\lambda_j$.
\end{proof}

\section{Singular Value Decomposition (SVD) and Proper Orthogonal Decomposition (POD) for ergodic systems}
\label{sec:SVD-POD}
Singular Value Decomposition (SVD) is a central algorithm of linear algebra that lies at the heart of many data analysis techniques for dynamical systems including linear subspace identification methods \cite{van2012subspace} and DMD \cite{Schmid:2010,tu2014dynamic}. Proper Orthogonal Decomposition (POD), on the other hand, is a data analysis technique frequently used for complex and high-dimensional dynamical systems. Also known as Principal Component Analysis (PCA), or Karhonen-Loeve decomposition, POD yields an orthogonal basis for representing ensemble of observations which is optimal with respect to a pre-defined inner product.
It is known that for finite-dimensional observables on discrete-time dynamical systems, POD reduces to SVD \cite{holmes2012turbulence}. Here, we establish a slightly different connection between these two concepts in the case of ergodic systems.
Our motivation for derivation of these results is the role of SVD in DMD algorithms, however, the orthogonal basis that is generated by this process can be used for further analysis of dynamics in the space of observables, for example, to construct a basis for computing the eigenfunctions of the Koopman generator as in \cite{giannakis2015data}. We first review POD and then record our main result in \cref{prop:SVD-POD}.

Let $\mathcal{F}=[f_1,~f_2,~\ldots,~f_n]$ be an ensemble of observables in the Hilbert space $\h$ which spans a $k-$dimensional subspace. Applying POD to $\mathcal F$ yields the expansion
\begin{align}
\mathcal{F}&=\Psi \Sigma V^*, \label{eq:POD}
\\
           &=[\psi_1,\psi_2,\ldots,\psi_k]
\left[
\begin{matrix}
\s_1 &  0  & \ldots & 0\\
0 &  \s_2 & \ldots & 0 \\
\vdots & \vdots & \ddots &   \vdots\\
0  &   0       &\ldots & \s_k
\end{matrix}\right]
\left[
  \begin{array}{ccc}
    \horzbar & v^{*}_{1} & \horzbar \\
    \horzbar & v^{*}_{2} & \horzbar \\
             & \vdots    &          \\
    \horzbar & v^{*}_{k} & \horzbar
  \end{array}
\right],
\nonumber
\end{align}
where $\psi_j$'s, $j=1,2,\ldots,k$ form an orthonormal basis for $\spn\{\mathcal F\}$, and are often called the \emph{empirical orthogonal functions} or \textit{POD basis} of $\mathcal{F}$.
The diagonal elements of $\Sigma$, denoted by $\sigma_j,j=1,2,\ldots,k$ are all positive and signify the $\h$-norm contribution of the basis element $\psi_j$ to the $\mathcal{F}$. The columns of $V$, denoted by $v_i$ and called the \emph{principal coordinates}, are the normalized coordinates of vectors in $\mathcal{F}$ with respect to the POD basis. This decomposition can be alternatively written as a summation,
\begin{align}
f_i= \sum_{j=1}^{k} \sigma_j \psi_j v_{ji}.\label{eq:PODsum}
\end{align}
If we index the principal coordinates such that $\sigma_1>\sigma_2>\ldots>\sigma_k>0$, then this decomposition minimizes the expression
\begin{align*}
e_p= \frac{1}{n}\sum_{i=1}^{n} \left\|\sum_{j=1}^{p} \sigma_j \psi_j v_{ji}-f_i\right\|_\h
\end{align*}
for any $p\leq k$, over the choice of all orthonormal bases for $\spn\{\mathcal{F}\}$. The term $e_p$ denotes the average error in approximating $f_i$'s by truncating the sum in \cref{eq:PODsum} at length $p$. This property, by design, guarantees that low-dimensional representations of $\mathcal{F}$ using truncations of POD involves the least $\|\cdot\|_\h$-error compared to other choices of orthogonal decomposition.

An established method for computation of POD is \emph{the method of snapshots} \cite{sirovich1987turbulence}: we first form the Gramian matrix $G$, given by $G_{ij}=<f_i,f_j>_\h$. The columns of $V$ are given as the normalized eigenvectors of $G$ associated with its non-zero eigenvalues, and those non-zero eigenvalues happen to be $\s_i^2$'s, that is,
\begin{align}
  G V=V\Sigma^2. \label{eq:PODG}
\end{align}
Since $G$ is a symmetric real matrix, $V$ will be an orthonormal matrix and the decomposition in \cref{eq:POD} can be easily inverted to yield the orthonormal basis functions,
  \begin{align*}
    \psi_j=\frac{1}{\sigma_j}\mathcal{F}v_j.
  \end{align*}
The SVD of tall rectangular matrix has a similar structure to POD. Consider $X_{m\times n}$, with $m>n$, to be a matrix of rank $r$. The \emph{reduced SVD} of $X$ is
\begin{align}
X& =WS\tilde V^* \label{eq:StandardSVD}\\
  &=
  \left[
  \begin{array}{cccc}
  \vertbar & \vertbar &        & \vertbar \\
  w_{1}    & w_{2}    & \ldots & w_{r}    \\
  \vertbar & \vertbar &        & \vertbar
  \end{array}
  \right]
  \left[
  \begin{matrix}
  s_1 &  0  & \ldots & 0\\
  0 &  s_2 & \ldots & 0 \\
  \vdots & \vdots & \ddots &   \vdots\\
  0  &   0       &\ldots & s_r
  \end{matrix}\right]
  \left[
  \begin{array}{ccc}
  \horzbar & \tilde v^{*}_{1} & \horzbar \\
  \horzbar & \tilde v^{*}_{2} & \horzbar \\
  & \vdots    &          \\
  \horzbar & \tilde v^{*}_{r} & \horzbar
  \end{array}
  \right]\nonumber
\end{align}
where $W$ and $V$ are orthonormal matrices, and $S$ is a diagonal matrix holding the  singular values $s_1>s_2>\ldots>s_r>0$. The columns of $W$ and $V$ are called, respectively, left and right singular vectors of $X$.

The usual practice of POD in data analysis is to let $\h$ be the space of snapshots, e.g. $\mathbb{R}^m$ equipped with the usual Euclidean inner product, which makes POD and SVD identical \cite{holmes2012turbulence}. In that case, $\mathcal{F}$ would be a  snapshot matrix such as \cref{eq:DataMatrix} and its left singular vectors are the POD basis.
We are interested in $\h$, however, as the infinite-dimensional space of observables defined in \cref{eq:HilbertDef}. Given the sampling of a set of observables on a single ergodic trajectory, our computational goal is to get the sampling of the orthonormal basis functions for the subspace of $\h$ spanned by those observables. The next proposition shows that this can be achieved by applying SVD to a data matrix whose columns are ergodic sampling of those observables. Incidentally, such matrix is the transpose of a snapshot matrix!


\begin{proposition}[\textbf{Convergence of SVD to POD for ergodic systems}]\label{prop:SVD-POD}

Let $\mathcal{F}=[f_1,~f_2,~\ldots,~f_n]$ be an ensemble of observables on the ergodic dynamical system in \cref{discDSGen}. Assume $\mathcal{F}$ spans a $k-$dimensional subspace of $\h$ and let
  \begin{align*}
    \mathcal{F} = \Psi \Sigma V^*,
  \end{align*}
 be the POD of $\mathcal{F}$. Now consider the data matrix,
    \begin{align}
\tilde F = \left(
\begin{matrix}
f_1(z_0) &  f_2(z_0)  & \ldots & f_n(z_0)\\
f_1\circ T(z_0) &  f_2\circ T(z_0) & \ldots & f_n \circ T(z_0)\\
\vdots & \vdots & \ddots &   \vdots\\
f_1\circ T^{m-1}(z_0)  &   f_2\circ T^{m-1}(z_0)       &\ldots & f_n \circ T(z_0)
\end{matrix}\right)
 \label{eq:Ftilde}
\end{align}
and let
\begin{align*}
  \frac{1}{\sqrt{m}}\tilde F \approx  W   S  V^*
\end{align*}
be the reduced SVD of $(1/\sqrt{m})\tilde F$.

Then, for almost every $z_0$, as $m\rar \infty$,
\newline (a) $ s_j\rar \sigma_j$ for $j=1,2,\ldots,k$ and $ s_j\rar 0$ for $ j=k+1,\ldots,n$.
\newline (b) $\tilde{v}_j \rar v_j$ for $j=1,2,\ldots,k$.
\newline (c) $\sqrt{m}w_j$ converges to the sampling of $\psi_j$ along the trajectory starting at $z_0$ $j=1,2,\ldots,k$.
\end{proposition}

\begin{proof} Consider the numerical Gramian matrix
  $$\tilde G:= (\frac{1}{\sqrt{m}}\tilde F)^*(\frac{1}{\sqrt{m}}\tilde F)= \frac{1}{m} \tilde F^* \tilde F.$$
 As shown in the proof of \cref{prop:Hankel-DMD}, the assumption of ergodicity implies the convergence of $\tilde G$ to the Gramian matrix $G$ in \cref{eq:PODG} as $m\rar \infty$.
 Now denote by $\lambda_j,~j=1,2,\ldots,k$, the $k$ eigenvalues of $\tilde G$ that converge to positive eigenvalues of $G$, and denote by $\lambda'_j,~j=1,\ldots,n-k$, the eigenvalues of $\tilde G$ that converge to zero. There exists $m_0$ such that for any $m>m_0$, all the singular values $s_j=\sqrt {\lambda_j}$ are larger than $s'_j=\sqrt {\lambda'_j}$ and therefore occupy the first $k$ diagonal entries of $S$, which completes the proof of statement in (a).

The statement in (b) follows from the convergence of the normalized eigenvectors of $\tilde G$ associated with $\lambda_j$'s to those of $G$.
To show that (c) is true, we first construct the candidate functions by letting
\begin{align}
    \tilde{\psi}_j =\frac{1}{s_j}\mathcal{F}\tilde{v}_j,~j=1,2,\ldots,k \label{eq:Psican}
 \end{align}
We compute the entries of the orthogonal projection matrix $P$ defined by $P_{i,j}=<\psi_i,\tilde \psi_j>$ and consider its limit as $m\rar\infty$,
  \begin{align}\label{eq:matrixP}
   \lim_{m\rar\infty} P_{i,j}& =\lim_{m\rar\infty} (\frac{1}{\s_i}\mathcal{F}v_i)^*(\frac{1}{s_j}\mathcal{F}\tilde{v}_j)
   =  \frac{1}{\s_i}v^*_i\mathcal{F}^*\mathcal{F}\lim_{m\rar\infty}\frac{1}{s_j}\tilde{v}_j
  = \frac{1}{\s^2_i}v^*_iGv_j=\delta_{ij}
  \end{align}
In the last equality, we have used \cref{eq:PODG}. This calculation shows that $\tilde \psi_j$'s are weakly convergent to $\psi_j$'s, and since they belong to the same finite-dimensional space, that implies strong convergence as well.
Noting the definition of left singular vector
  \begin{align}
    w_j = \frac{1}{\sqrt{m}s_j} \tilde F v_j, \label{eq:W}
  \end{align}
  it is easy to see that $\sqrt{m}w_j$ is the sampling of the candidate function $\tilde {\psi}_j$ along the trajectory starting at $z_0$.
\end{proof}


\subsection{Representation of chaotic dynamics in space of observables using Hankel matrix and POD}
Using the above results, we are able to construct an orthonormal basis on the state space from the time series of an ergodic system and give a new data-driven representation of chaotic dynamical systems.
Recall that the columns of the Hankel matrix $\tilde H$ defined in \cref{eq:Hankel} provides an ergodic sampling of the Krylov sequence of observables  $\mathcal{F}_n=[f,~Uf,~,U^{n}f]$ along a trajectory in the state space.
Therefore by applying SVD to $\tilde H$, we can approximate an orthonormal basis for $\mathcal{F}_n$, and furthermore, represent the Koopman evolution of observable $f$ in the form of principal coordinates.

We show an example of this approach using the well-known chaotic attractor of the Lorenz system \cite{lorenz:1963}. This attractor is proven to have the mixing property which implies ergodicity \cite{luzzatto2005lorenz}.
The Lorenz system is given by
\begin{align*}
\dot{z}_1 &= \sigma (z_2-z_1),\\
\dot{z}_2 &= z_1 (\rho-z_3)-z_2,\\
\dot{z}_3 &= z_1z_2-\beta z_3,
\end{align*}
with $[z_1,z_2,z_3]\in\mathbb{R}^3$ and parameter values $\s=10,~\rho=28$ and $\beta=8/3$. First, we sample the value of observable $f(\z)=z_1$ every $0.01$ seconds on a random trajectory, and then form a tall Hankel matrix as in \cref{eq:Hankel} with $m=10000$ and  $n=500$.
 \Cref{fig:LorenzPODbasis} shows the values of first six left singular vectors of the Hankel matrix, which approximate the basis functions $\psi_j$. The corresponding right singular vectors, shown in  \cref{fig:LorenzPODcoordinates}, approximate the principal coordinates of $\mathcal{F}_n$.
We make note that the computed basis functions and their associated singular values show little change with $m$ for $m\geq10000$.

\begin{figure}[ht]
	\centerline{\includegraphics[width=1.1 \textwidth]{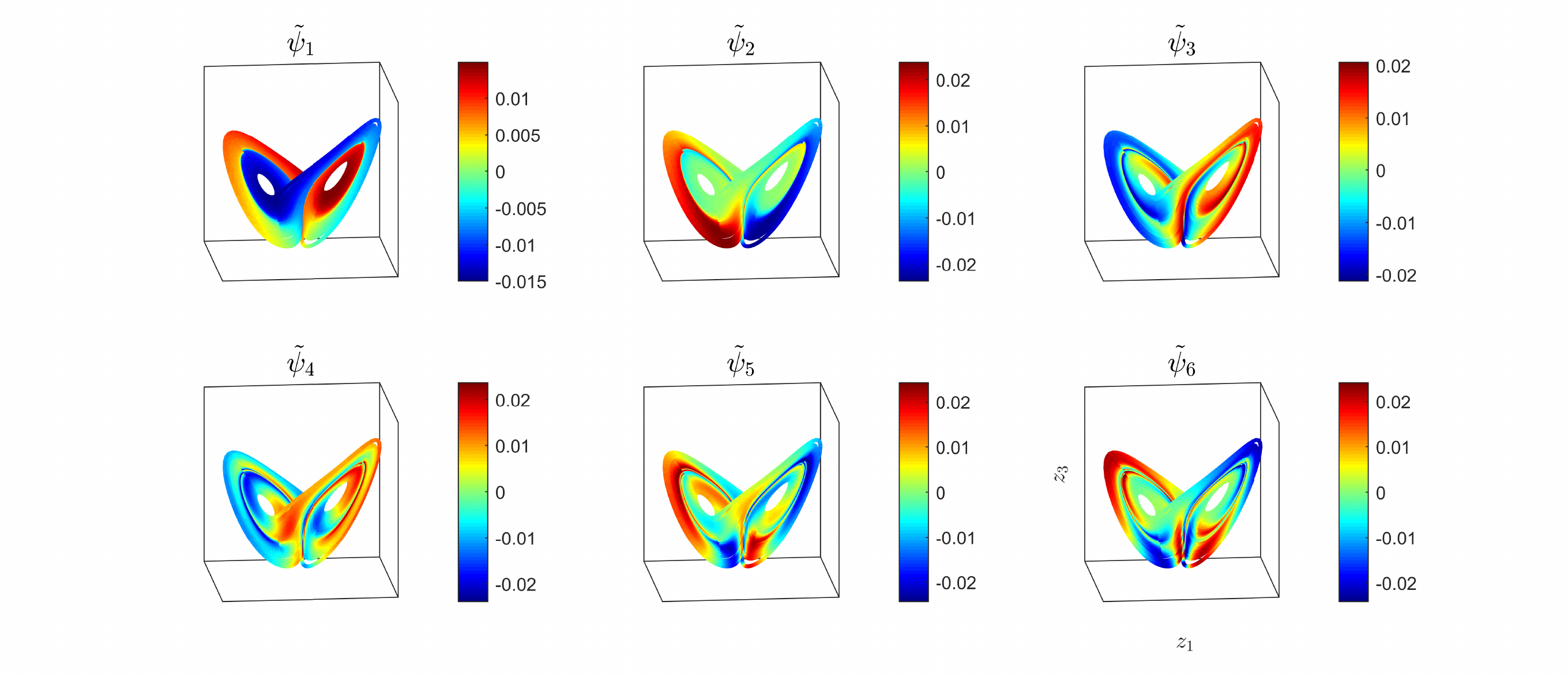}}
	\caption{The first six POD basis functions of $[f,~Uf,\ldots,U^if,\ldots,U^{500}f]$. The observable is $f(z_1,z_2,z_3)=z_1$ and the Hankel matrix has the dimensions $m=10000$ and $n=500$. }
	\label{fig:LorenzPODbasis}
\end{figure}

\begin{figure}[ht]
	\centerline{\includegraphics[width=1.1 \textwidth]{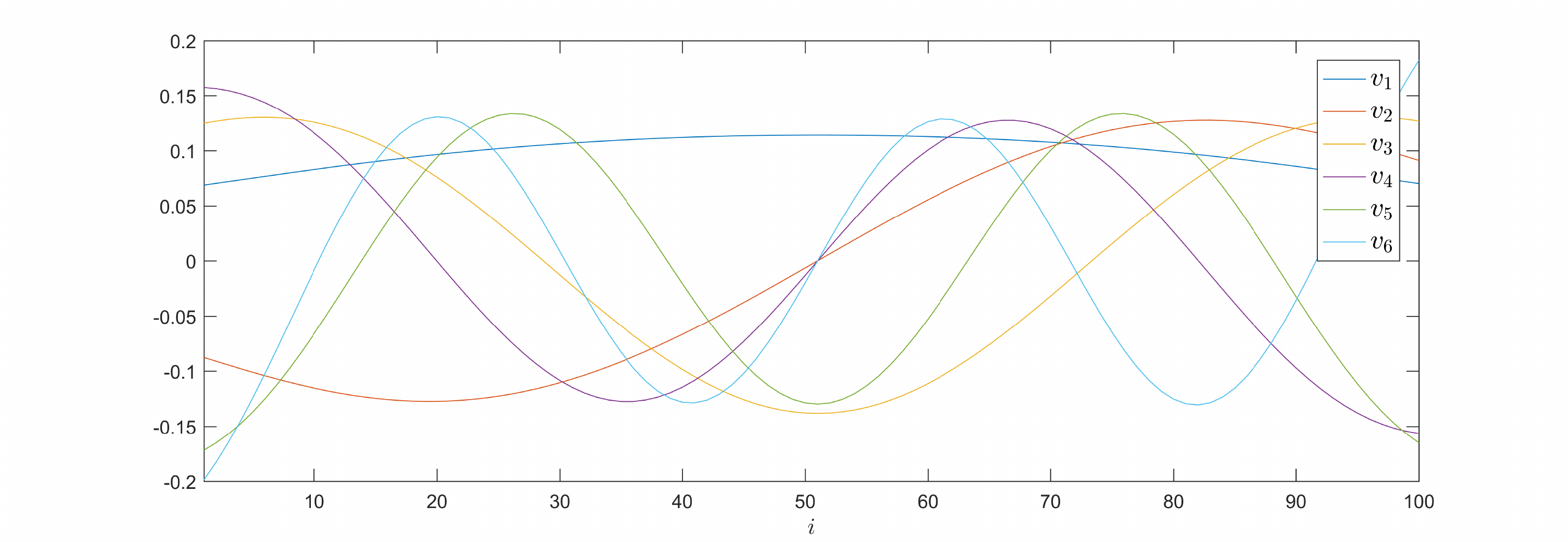}}
	\caption{The principal coordinates for $[f,~Uf,\ldots,U^if,\ldots,U^{500}f]$ with $f(z_1,z_2,z_3)=z_1$ in the basis shown in \cref{fig:LorenzPODbasis}. }
	\label{fig:LorenzPODcoordinates}
\end{figure}

In mixing attractors such as Lorenz, the only discrete eigenvalue for the Koopman operator is $\lambda=1$, which is associated with eigenfunction that is constant almost everywhere on the attractor. In this case, the Koopman operator cannot have any invariant finite-dimensional subspace other than span of almost-everywhere constant functions.
As a result, for a typical observable $f$, the Krylov sequence $\mathcal{F}_n$ is always $n+1$-dimensional and growing with iterations of $U$. This observation shows that despite the fact that evolution of principal coordinates in \cref{fig:LorenzPODcoordinates} is linear, there is no finite-dimensional linear system that can describe their evolution.

\Cref{prop:SVD-POD} offers a lifting of coordinates from state space to the space of observables which can be useful in the study of chaotic systems. Application of SVD to embedded time series in the form of Hankel matrix is a popular technique in the study of climate time series under the name of Singular Spectrum Analysis (SSA) \cite{ghil2002advanced}. This method is frequently used for pattern extraction and spectral filtering of short and noisy time series. In a more recent example, Brunton and co-workers \cite{brunton2016chaos} have introduced a new framework based on the SVD of Hankel data matrix
to construct  a linear representation of several chaotic systems including Lorenz attractor. They have discovered that the evolution of principal coordinates in some classes of chaotic systems - including Lorenz attractor -  can be described via a low-dimensional linear system with intermittent forcing. By using such model, they were able to predict the highly nonlinear behavior of those systems over short time windows given the knowledge of the forcing. This suggests that replacing the state-space trajectory-based analysis with the evolution of coordinates in space of observable gives a more robust representation of the dynamics for analysis and control purposes.

\section{Convergence of Exact DMD and extension to multiple observables} \label{sec:ExactDMD}
In this section, we review the functional setting for Exact DMD and prove its convergence for ergodic systems using the assumption of invariant subspace. We then discuss using this method combined with Hankel matrices to compute the Koopman spectra from observations on single or multiple observables. The summary of the numerical algorithm with several examples will be given in the next section.

Let $\mathcal{F}=:[f_1,f_2,\ldots,f_n]$ denote a set of observables defined on the discrete dynamical system in \cref{discDSGen}. We make the assumption that $\mathcal{F}$ spans a $k-$dimensional subspace of $\h$, with $k\leq n$, which is invariant under the Koopman operator.
Also denote by $U\mathcal{F}=:[Uf_1,Uf_2,\ldots,Uf_n]$ the image set of those observables under the action of the Koopman operator. We seek to realize the Koopman operator, restricted to this subspace, as a $k$-by-$k$ matrix, given the knowledge of $\mathcal{F}$ and $U\mathcal{F}$.

Let
\begin{align}
  \mathcal{F}=\Psi \Sigma V^*. \label{eq:F1POD}
\end{align}
be the POD of ensemble $\mathcal{F}$ with $\Psi=[\psi_1,~\psi_2,\ldots,~\psi_k]$ denoting its POD basis.
Since $\mathcal{F}$ spans an invariant subspace, the functions in $U\mathcal{F}$ also belong to the same subspace and their principal coordinates can be obtained by orthonormal projection, i.e.,
\begin{align}
  \Omega=\Psi^*U\mathcal{F}. \label{eq:Omega}
\end{align}
The restriction of the Koopman operator to the invariant subspace is then given by matrix $A$ which maps the columns of $\Sigma V^*$ to the columns of $\Omega$. The following lemma, which summarizes some of the results in \cite{tu2014dynamic}, gives the explicit form of matrix $A$ and asserts its uniqueness under the prescribed condition on $\mathcal{F}$ and $U\mathcal{F}$.

\begin{lemma}
  Let $X_{k\times n}$ with $n\geq k$ be a matrix whose range is equal to $\mathbb{R}^k$. Let $Y_{k\times n}$ be another matrix which is linearly consistent with $X$, i.e., whenever $Xc=0$ for $c\in \mathbb{R}^n$, then $Yc=0$ as well. Then the Exact DMD operator $A:=YX^{\dagger}$ is the unique  matrix that satisfies $AX=Y$.
\end{lemma}
\begin{proof}
  First, we note that condition of $Y$ being linearly consistent with $X$ implies that $A$ satisfies $AX=Y$ (theorem 2 in \cite{tu2014dynamic}).
  To see the uniqueness, let $\tilde A$ be another matrix which satisfies $\tilde A X=Y$.
  Now let $b\in \mathbb{R}^k$ be an arbitrary vector. Given that $X$ spans $\mathbb{R}^k$, we can write $b=Xd$ for some $d\in \mathbb{R}^n$. Consequently,
  \begin{align*}
    \tilde A b = \tilde AXd = Yd = A Xd= A b,
  \end{align*}
  which means that action of $A$ and $\tilde A$ on all elements of $\mathbb{R}^k$ is the same, therefore $\tilde A=A$.
\end{proof}
Since $\mathcal{F}$ and $U\mathcal{F}$ are related through a linear operator, their principal coordinates with respect to the same orthogonal basis satisfy the condition of linear consistency, and the Koopman operator restricted to the invariant subspace is represented by the matrix
\begin{align}
  A = \Omega (\Sigma V^*)^\dagger = \Omega V\Sigma^{-1}, \label{eq:ExDMDOp}
\end{align}
where we have used the fact that $\Sigma$ is diagonal and $V$ is orthonormal. Let $(w_j,\lambda_j),~j=1,2,\ldots,k$ denote the eigenvector-eigenvalue pairs of $A$. Then $\lambda_j$'s are the Koopman eigenvalues and the associated Koopman eigenfunctions are given by
\begin{align}
  \phi_j = \Psi w_j,\quad j=1,2,\ldots,k.
\end{align}
Now we can assert the convergence of Exact DMD projected modes and eigenvalues to Koopman eigenfunctions and eigenvalues given the ergodic sampling of functions in $\mathcal{F}$ and $U\mathcal{F}$.

\begin{proposition}[\textbf{Convergence of Exact DMD for ergodic sampling}]\label{prop:ExactDMD}
  Let $\mathcal{F}=:[f_1,f_2,\ldots,f_n]$ denote a set of observables that span a $k-$dimensional invariant subspace of the Koopman operator with $k\leq n$, defined on the dynamical system \cref{discDSGen} which is ergodic. Now consider the data matrices
  \begin{align*}
    X=\left[
    \begin{matrix}
      f_1(z_0) & f_2(z_0) & \ldots & f_n(z_0)\\
      f_1\circ T(z_0) & f_2\circ T(z_0) & \ldots & f_n\circ T(z_0)\\
      \vdots & \vdots & \ddots & \vdots\\
      f_1\circ T^{m-1}(z_0) & f_2\circ T^{m-1}(z_0) & \ldots & f_n\circ T^{m-1}(z_0)\\
    \end{matrix}
      \right]
  \end{align*}
and
  \begin{align*}
    Y=\left[
    \begin{matrix}
      f_1\circ T(z_0) & f_2\circ T(z_0) & \ldots & f_n\circ T(z_0)\\
      f_1\circ T^2(z_0) & f_2\circ T^2(z_0) & \ldots & f_n\circ T^2(z_0)\\
      \vdots & \vdots & \ddots & \vdots\\
      f_1\circ T^m(z_0) & f_2\circ T^m(z_0) & \ldots & f_n\circ T^m(z_0)\\
    \end{matrix}
      \right]
  \end{align*}
   Let $\lambda_j$ and $\chi_j$ denote the dynamic eigenvalues and projected dynamic modes, respectively, obtained by applying Exact DMD (\cref{alg:EDMD}) to $X$ and $Y$, using a $k$-dimensional truncation of SVD (i.e. with discarding the $n-k$ smallest singular values) in step 2.

  Then, for almost every $z_0$, as $m\rar\infty$:
\newline (a) The dynamic eigenvalues converge to the Koopman eigenvalues.
\newline (b) $\sqrt{m}\chi_j$ for $j=1,2,\ldots,k$ converge to the sampling of Koopman eigenfunctions along the trajectory starting at $z_0$.

\end{proposition}
\begin{proof}
  We first show that as $m\rar \infty$, the matrix $\tilde A$ constructed in the step 3 of the \cref{alg:EDMD} converges to the matrix $A$ in \cref{eq:ExDMDOp}. Let $X=WS\tilde V$ be the $k$-dimensional truncated SVD of $X$. It follows from the proof of \cref{prop:SVD-POD} that  $ S /\sqrt{m} \rar \Sigma$ and $\tilde V\rar V$ as $m\rar \infty$. Then
  \begin{align}
    \lim_{m\rar\infty}\tilde A &= \lim_{m\rar\infty} W^*Y\tilde V S^{-1}= \left(\lim_{m\rar\infty}\frac{1}{\sqrt{m}} W^*Y \right) V \Sigma^{-1}. \label{eq:tildeA1}
  \end{align}

We need to show that $\tilde \Omega :=1/\sqrt{m}W^*Y$ converges to $\Omega$ (defined in \cref{eq:Omega}) as $m\rar \infty$. To this end, recall the candidate functions in $\tilde \Psi$ defined in proof of \cref{prop:SVD-POD}. We see that
\begin{align}
  \tilde \Omega_{i,j} =  \frac{1}{\sqrt{m}}W_i^*Y_j =  \frac{1}{m}\sum_{l=0}^{m-1} (\tilde \psi_i \circ T^l(z_0))^* (Uf_j\circ T^l(z_0) ):= \frac{1}{m}\sum_{l=0}^{m-1} \tilde \psi_i^{l*} Uf_j^l \label{eq:OmegaTilde}
\end{align}
Since $\tilde\psi_i$'s are defined by linear combination of functions in $\mathcal{F}$, they lie in the span of $\Psi$. In fact, we have
\begin{align}
  \tilde \psi_{i} =\sum_{q=1}^k P_{i,q} \psi_q, \label{eq:PsiTExpansion}
\end{align}
where $P_{i,q}$ denotes an entry of the orthogonal projection matrix $P$ defined in the proof of \cref{prop:SVD-POD}. The computation in \cref{eq:matrixP} shows that as $m\rar\infty$, we have $P_{i,q}\rar 0$ if $q\neq i$ and $P_{i,i}\rar1$. Now we replace $\tilde \psi_i$ in \cref{eq:OmegaTilde} with its expansion in \cref{eq:PsiTExpansion}, while considering the term
\begin{align}
  d_1:&= \tilde\Omega_{i,j}-\frac{1}{m}\sum_{l=0}^{m-1} ( \psi_i \circ T^l(z_0))^* (Uf_j\circ T^l(z_0) ), \\
  &:= \tilde\Omega_{i,j}-\frac{1}{m}\sum_{l=0}^{m-1} \psi_i^{l*} Uf_j^l , \nonumber\\
  &= \sum_{q=1}^k P_{i,q} \frac{1}{m}\sum_{l=0}^{m-1} \psi_q^{l*} Uf_j^l - \frac{1}{m}\sum_{l=0}^{m-1} \psi_i^{l*} Uf_j^l \nonumber , \nonumber\\
    &= \sum_{q=1,~q\neq i}^k P_{i,q} \frac{1}{m}\sum_{l=0}^{m-1} \psi_q^{l*} Uf_j^l + (P_{i,i}-1)\frac{1}{m}\sum_{l=0}^{m-1} \psi_i^{l*} Uf_j^l. \nonumber
\end{align}
The convergence of sums over $l$ for $m\rar\infty$ in the last line is given by \cref{eq:innerproduct}. Combining that with the convergence of $P_{i,q}$'s, we can conclude that for every $\epsilon>0$, there exists $m_1$ such that for any $m>m_1$, we have $|d_1|<\epsilon/2$. Now consider the term
\begin{align}
d_2 &= \frac{1}{m}\sum_{l=0}^{m-1} \psi_i^{l*} Uf_j^l - \Omega_{i,j} \\
	&= \frac{1}{m}\sum_{l=0}^{m-1} \psi_i^{l*} Uf_j^l -<\psi_i,Uf_j> . \nonumber
\end{align}
The convergence in \cref{eq:innerproduct} again implies that for every $\epsilon>0$ there exists an $m_2$ such that for any $m>m_2$, we have $|d_2|<\epsilon/2$. Now it becomes clear that
\begin{align}
|\tilde\Omega_{i,j}-\Omega_{i,j}|=|d_1+d_2|\leq |d_1|+|d_2| = \epsilon,
\end{align}
for any $m>\max(m_1,m_2)$. This proves the convergence of $\tilde \Omega$ to $\Omega$, which, by revisiting \cref{eq:tildeA1}, means
\begin{align}
\lim_{m\rar\infty} \tilde A = A.
\end{align}

The eigenvalues of $\tilde A$ converge to the eigenvalues of $A$ which are the Koopman eigenvalues.
Let $\tilde w_j,~j=1,2,\ldots,k$ denote the normalized eigenvectors of $\tilde A$, and define the candidate functions $\tilde \phi_j = \tilde \Psi \tilde w_j = \Psi P \tilde w_j $.
Given that $P$ converges to $I$, and $\tilde w_j$ converges to $w_j$ as $m\rar\infty$, it follows that $\tilde \phi_j$ strongly converge to the Koopman eigenfunctions $\phi_j=\Psi w_j$.

Now note that $\sqrt{m}\chi_j=\sqrt{m}U\tilde w_j$ is the sampling of candidate eigenfunction $\tilde\phi_j$ along the trajectory, and therefore it converges to the sampling of the Koopman eigenfunction $\phi_j$ as $m\rar\infty$.
\end{proof}

{
Recall that in \cref{sec:VanderPol}, we extended the application of companion-matrix Hankel-DMD to the basin of attraction of an ergodic set  $A$ by assuming the invariant measure on $A$ is physical. An analogous extension of the above proposition can be stated using the same assumption:

\begin{corollary}\label{cor:EDMDbasin}
  \Cref{prop:ExactDMD} is also valid for $\nu-$almost every $z_0\in \mathcal{B}$ ($\nu$ is the Lebesgue measure), where $\mathcal{B}$ is the basin of attraction for the ergodic attractor $A$ of the dynamical system \cref{discDSGen}, given that
  \newline (i) The invariant measure on $A$ is a physical measure,
  \newline (ii) $\mathcal{F}|_A$ spans a $k$-dimensional invaraint subspace of $\h$, and
  \newline (iii) $\mathcal{F}$ is continuous over $\mathcal{B}$.
\end{corollary}
\begin{proof}
  The proof is similar to \cref{prop:Hankel-DMD-basin}.
\end{proof}
}
Let us first consider the application of the above proposition using measurements on a single observable $f$. We can supplement those measurements using delay-embedding which is equivalent to setting $\mathcal{F}=[f,~Uf,\ldots,U^{n-1}f]$.
An ergodic sampling of $\mathcal{F}$ and $U\mathcal{F}$ is then given by data matrices
\begin{align*}
  X=\tilde H \qquad Y=U\tilde H
\end{align*}
where $\tilde H$ is the Hankel matrix defined in \cref{eq:Hankel} and $U\tilde H$ is the same matrix but shifted one step forward in time.
In such case, the Exact DMD reduces to the SVD-enhanced DMD as discussed in \cref{sec:IntroDMD}.

In case of multiple observables, we can combine the delay-embedded measurements of the observables with each other.
For example, let $f$ and $g$ be the only observables that could be measured on a dynamical system. Then, we let  $F_1=[f,~Uf,\ldots,~U^{l-1}f,~g,~Ug,\ldots,~U^{q-1}g]$ and data matrices would contain blocks of Hankel matrices, i.e.,
	\begin{align}
	X=\left[\begin{matrix}
		\tilde H_f & \tilde H_g
	\end{matrix}\right] ,\quad
		Y=\left[\begin{matrix}
		U\tilde H_f & U\tilde H_g
		\end{matrix}\right],
	\end{align}
The above proposition guarantees the convergence of the Exact DMD method if $l$ and $q$ are chosen large enough, e.g., $l,q>k+1$ where $k$ is the dimension of the invariant subspace containing $f$ and $g$.

In numerical practice, however, the block Hankel matrices need some scaling. For instance assume $\|g\|_\h\ll \|f\|_\h$. The POD basis that corresponds to the measurements on $g$ is associated with smaller singular values and might be discarded through a low-dimensional SVD truncation. To remedy this issue, we can use the fact that the ratio of the norm between observables in ergodic systems can be approximated from the measurements:
\begin{align}
\alpha:=\frac{\|f\|_\h}{\|g\|_\h} = \lim_{m\rar\infty} \frac{\|\tilde f_m\|}{\| \tilde g_m\| },
\end{align}
where $\tilde f_m$ and $\tilde g_m$ are the observation vectors defined in \cref{eq:observationVector}. The scaled data matrices in that case become
	\begin{align}
	X=\left[\begin{matrix}
	\tilde H_f & \alpha \tilde H_g
	\end{matrix}\right] ,\quad
	Y=\left[\begin{matrix}
	U\tilde H_f & \alpha U\tilde H_g
	\end{matrix}\right].
	\end{align}

\section{Numerical application of Hankel-DMD method}
\label{sec:applications}
\Cref{alg:HDMD} summarizes the Hankel-DMD method for extracting the Koopman spectrum from single or multiple observables.
This algorithm acts on Hankel matrices of the data in the form of
\begin{align}
			\tilde H^i &= \left(
			\begin{matrix}
				f_i(z_i) &  f_i\circ T(z_i)  & \ldots & f_i\circ T^n(z_i)\\
				f_i\circ T(z_i) &  f_i\circ T^2(z_i) & \ldots & f_i\circ T^{n+1}(z_i)\\
				\vdots & \vdots & \ddots &   \vdots\\
				f_i\circ T^{m-1}(z_i)  &   f_i\circ T^{m}(z_i)       &\ldots & f_i\circ T^{m+n-1}(z_i) \\
			\end{matrix}\right),
			\quad i=1,2,\ldots l, \label{eq:HData}
\end{align}
and $U\tilde H^i$ which is the same matrix shifted forward in time.
The data on observable $f_i$ is collected from a trajectory starting at $z_i$ which is in the basin of attraction of an ergodic attractor. Unlike the classical DMD algorithms in \cref{sec:IntroDMD}, the number of modes obtained by this method depends on the length of the signal and the dimension of subspace in which the observable lies.

  The rate of convergence for the Hankel-DMD can be established by considering the rate of convergence for ergodic averages. For periodic and quasi-periodic attractors, the error of approximating the inner products by \cref{eq:innerproduct} is generally bounded by $|c/m|$ for some $c\in \mathbb{R}$ \cite{mezic2002ergodic}. For strongly mixing systems, the rate of convergence slows down to $c/\sqrt{m}$. However, for the general class of ergodic systems convergence rates cannot be established \cite{krengel:1985}.
  As we will see in this section, a few hundred samples would be enough to determine the Koopman frequencies of periodic systems with great accuracy, while a few thousand would be enough for systems with a 2-torus attractor.

\begin{remark} \label{rmk:SVDthreshold}
   In proving the convergence of the Hankel-DMD algorithm, we have assumed the explicit knowledge of the dimension of the invariant subspace ($k$) that contains the observable. In numerical practice, $k$ can be found in the SVD step of the algorithm: \cref{prop:SVD-POD} showed that as $m\rar\infty$, the number of singular values converging to positive values is equal to $k$. Therefore, we can approximate the invariant subspace by counting the number of singular values that don't seem to decay to zero. We can implement this assumption in \cref{alg:HDMD} by hard-thresholding of the SVD in step 3, i.e., discarding the singular values that are smaller than a specified threshold. Such a threshold can be selected based on the desired numerical accuracy and considering the rate of convergence for ergodic averages which is discussed above.
 In case that the observable lies in an infinite-dimensional Koopman invariant subspace, we are going to assume that it lives in a finite-dimensional subspace down to a specific numerical accuracy. We can enforce this accuracy, again, by hard-thresholding the SVD.  In the following examples, we have chosen the hard threshold of SVD to be 1e-10.
\end{remark}

\begin{algorithm}
	\caption{Hankel DMD}
	\label{alg:HDMD}
	\begin{algorithmic}[1]
		\STATEx{Consider the Hankel matrices $\tilde H^i$'s defined in \cref{eq:HData}.}
		\STATE{Compute the scaling factors
			\begin{align*}
				\alpha_i = \frac{\|H^i_{n+1}\|}{\|H^1_{n+1}\|},\quad i=2,3,\ldots,l,
			\end{align*}
			where $H^i_{n+1}$ is the last column of $H^i$.
			}
				\STATE{Form the composite matrices	
					\begin{align*}
						X=\left[\begin{matrix}
							\tilde H_1 & \alpha_2\tilde H_2 &\ldots& \alpha_l\tilde H_l
						\end{matrix}\right] ,\quad
						Y=\left[\begin{matrix}
							U\tilde H_f & \alpha_2 U\tilde H_2& \ldots& \alpha_l U\tilde H_l
						\end{matrix}\right].
					\end{align*}			
					}
		\STATE{Compute the truncated SVD of $X$ (see \cref{rmk:SVDthreshold}):
			\begin{align*}
				X=WS\tilde V^*.
			\end{align*}}
			\STATE{ Form the matrix
				\begin{align*}
					\hat A=W^*Y\tilde V S^{-1}.
				\end{align*}}
				\STATE{Let $(\lambda_j,w_j),~j=1,2,\ldots,m$ be the eigenvalue-eigenvector pairs for $\hat A$. Then $\lambda_j$'s approximate the Koopman eigenvalues. }
					\STATE{ The dynamic modes $\chi_j$ given by
						\begin{align*}
							\chi_j=W w_j,\quad j=1,2,\ldots,m.
						\end{align*}
						approximate the Koopman eigenfunctions.}
					\end{algorithmic}
				\end{algorithm}

\subsection{Application to single observable: periodic and quasi-periodic cavity flow} \label{sec:Cavity}
In a previous study by the authors \cite{Hassan2016}, the lid-driven cavity flow was shown to exhibit periodic and quasi-periodic behavior at Reynolds numbers ($Re$) in the range of 10000-18000. The Koopman eigenvalues were computed by applying an adaptive combination of FFT and harmonic averaging to the discretized field of stream function, which is an observable with $\sim 4000$ values.
We use the Hankel-DMD method (\cref{alg:HDMD}) to extract the Koopman eigenvalues and eigenfunctions using a scalar valued observable (the kinetic energy) and compare our results with the following cases studied in \cite{Hassan2016}:

\begin{itemize}
	\item At $Re=13000$, the trajectory in the state space of the flow converges to a limit cycle with the basic frequency of $\omega_0=1.0042~rad/sec$. The Koopman frequencies in the decomposition of analytic observables are  multiples of the basic frequency, i.e., $k\omega_0,~k=0,1,2,\ldots$.
	\item At $Re=16000$, the post-transient flow is quasi-periodic with two basic frequencies $\omega_1=0.9762~rad/sec$ and $\omega_2=0.6089~rad/sec$. In this case, the flow trajectory wraps around a 2-torus in the state space of the flow and the Koopman frequencies are integral multiples of $\omega_1$ and $\omega_2$, that is $\omega = \mathbf{k}\cdot (\omega_1,\omega_2),~\mathbf{k}\in \mathbb{Z}^2$.
\end{itemize}

Let $\{E_i:=E(t_0+i\Delta t)\}$ denote the measurements on the kinetic energy of the flow at the time instants  $t_0+i\Delta t,~i=0,1,2,\ldots,s$. We first build the Hankel matrices of the kinetic energy observable,
\begin{align}
	\tilde{H}_E &= \left(
	\begin{matrix}
		E_0 &  E_1  & \ldots & E_n\\
		E_1 &  E_2 & \ldots &  E_{n+1}\\
		E_2 &  E_3 & \ldots &  E_{n+2}\\
		\vdots & \vdots & \ddots &   \vdots\\
		E_{m-1} &   E_m       &\ldots &  E_{m+n-1} \\
	\end{matrix}\right),
\quad
	U\tilde{H}_E &= \left(
	\begin{matrix}
		E_1 &  E_2  & \ldots & E_{n+1}\\
		E_2 &  E_3 & \ldots &  E_{n+2}\\
		E_3 &  E_4 & \ldots &  E_{n+3}\\
		\vdots & \vdots & \ddots &   \vdots\\
		E_{m} &   E_{m+1}       &\ldots &  E_{m+n} \\
	\end{matrix}\right).
	\label{eq:KEHankel}
\end{align}
and then apply \cref{alg:HDMD}. In computing the SVD of the Hankel matrix (step 2 in algorithm \cref{alg:SVDDMD}), we discard the singular values smaller than $1e-10$.
Due to the discrete-time nature of the measurements, the eigenvalues computed by this method correspond to the discrete map obtained by strobing the original continuous-time dynamical system at intervals of length $\Delta t$.
The eigenvalues $\lambda_j, j=1,2,\ldots$ computed via Hankel-DMD are related to the Koopman frequencies $\omega_j,~j=1,2,\ldots$ through the following:
\begin{align*}
	\lambda_j = e^{i\omega_j\Delta t}.
\end{align*}
The computed dynamic modes $\tilde \phi_j$ approximate the value of the associated Koopman eigenfunctions along the first $m$ points on the trajectory of the system.

For the periodic flow, we use $200$ samples of the kinetic energy signal, with sampling interval of $0.1~sec$, to form the above Hankel matrix with $m=n=100$. In table \cref{table:periodic}, we present frequencies obtained using the DMD-Hankel method, labeled by $\tilde \omega_k$, and compare with frequencies computed in \cite{Hassan2016}.
We also compare the computed eigenfunctions with the theory for periodic and quasi-periodic attractors presented in \cite{mezic2017koopman}. In fact, the Koopman eigenfunctions for periodic and quasi-periodic attractors are the Fourier basis in the time-linear coordinates defined on the attractor.  To make this notion precise, let $s\in [0,2\pi)$ be the parametrization of the limit cycle given by $\dot{s}=\omega_0$. The Koopman eigenfunction associated with the Koopman frequency $\omega_k:=k\omega_0$ is given by
\begin{align}
	\phi_k=e^{iks},\quad k=1,2,3,\ldots. \label{eq:limitcycleEF}
\end{align}
The real part of the Koopman eigenfunctions $\tilde \phi_k$ computed via the Hankel-DMD method are shown along the trajectory in figure \cref{fig:LinearRe13}. The mean squared error  in the approximation of the six eigenfunctions with largest vector energy is from an order of $10^{-5}$ or smaller ( last column in the table \cref{table:periodic}).

\begin{figure}
	\centerline{\includegraphics[width=.5 \textwidth]{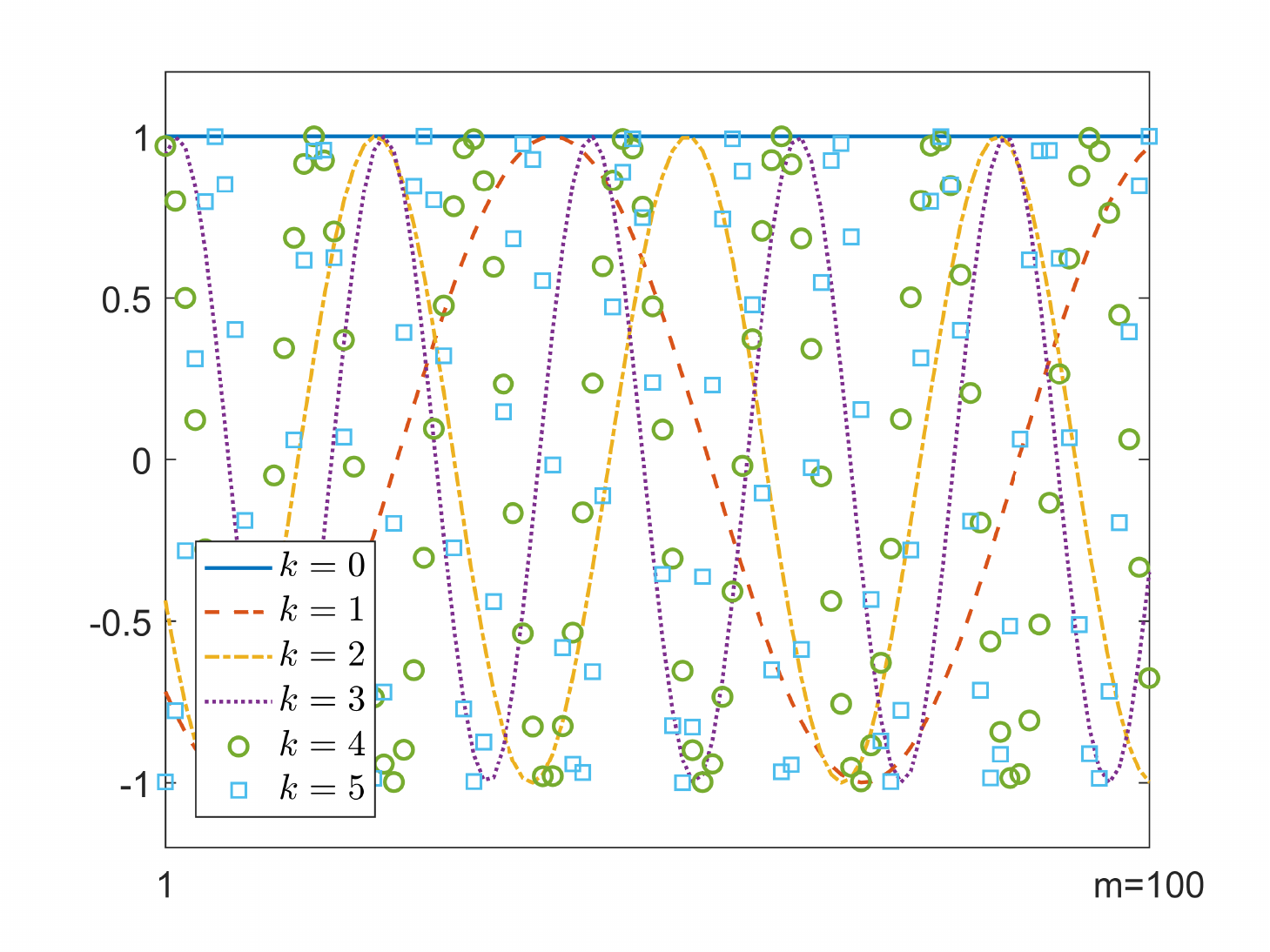}}
	\caption{Real part of the computed Koopman eigenfunctions $\tilde\phi_k$ along the trajectory for periodic cavity flow at $Re=13000$.}
	\label{fig:LinearRe13}
\end{figure}

\begin{table}[h]
	\centering 
	\begin{tabular}{c c c c c}
		k & $\tilde \omega^k$ & $\omega^k$ (\cite{Hassan2016}) & relative error  & $var(\tilde \phi_k -\phi_k)$ \\ [0.5ex]
		\hline 
		0 & 0 & 0 & 0 & $<1e{-10}$\\ 
		1 & 1.00421 & 1.00423 & $1.59e{-5}$ & $1.46e{-6}$\\
		2 & 2.00843 & 2.00840 & $1.56e{-5}$ & $2.25e{-5}$\\
		3 & 3.01264 & 3.01262  & $4.94e{-6}$& $<1.00e{-10}$ \\
		4 & 4.01685 & 4.01680 & $1.55e{-5}$& $<1.00e{-10}$ \\
		6 & 6.02528 & 6.02525 & $5.03e{-6}$ & $<1.00e{-10}$\\ [1ex] 
		\hline 
	\end{tabular}
	\caption{The dominant Koopman frequencies and eigenfunctions for periodic cavity flow computed using observations on kinetic energy.} 
	\label{table:periodic} 
\end{table}

\begin{figure}
	\centerline{\includegraphics[width=.4 \textwidth]{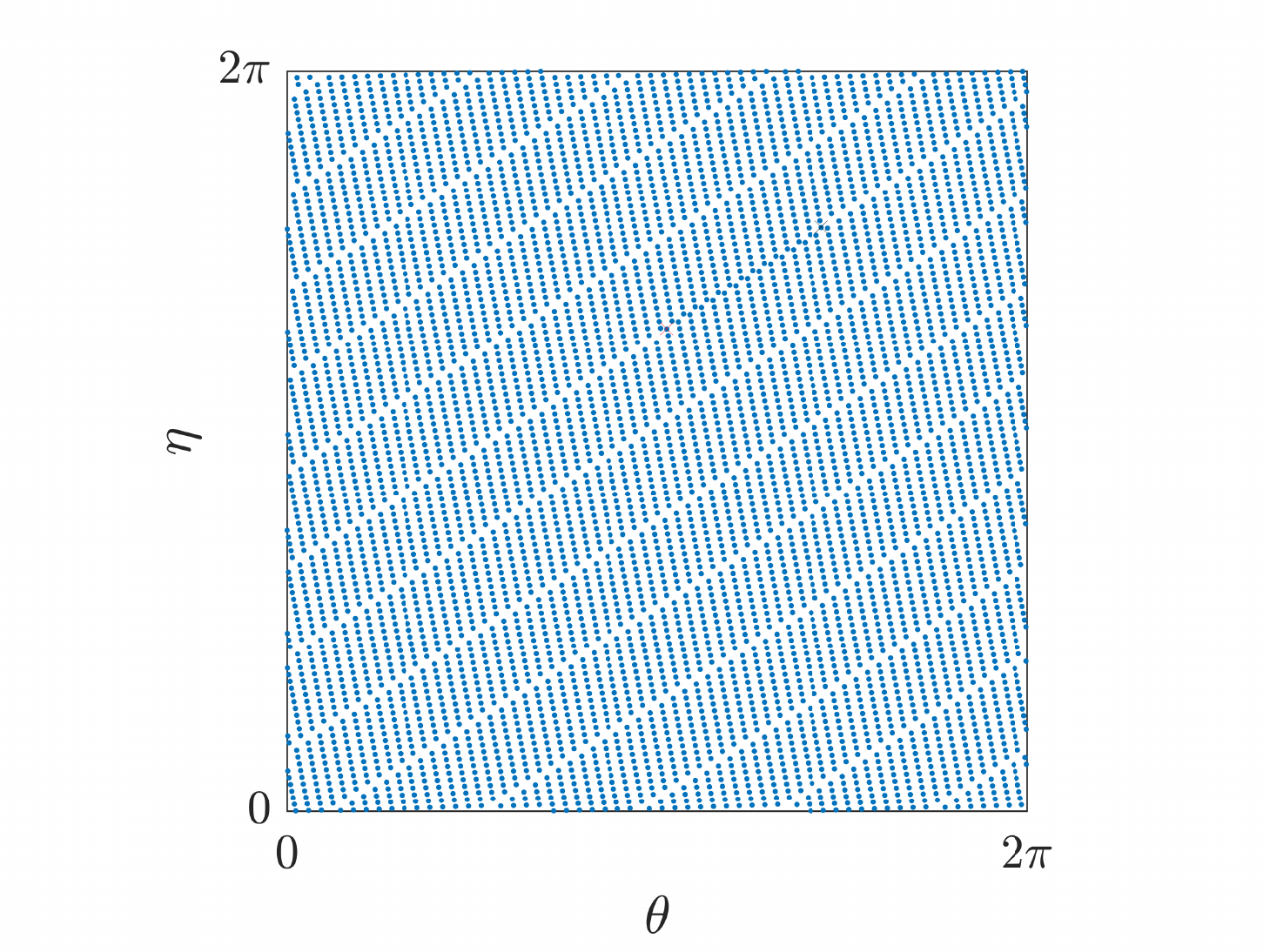}}
	\caption{The  trajectory of quasi-periodic flow on the parameterized torus defined in \cref{eq:lineartorus}.}
	\label{fig:Torus}
\end{figure}

For the quasi-periodic cavity flow, a longer sequence of observations is required to sufficiently sample the attractor which is a 2-torus. The eigenvalues shown in \cref{table:qperiodic} are computed using $6500$ samples of the kinetic energy signal with sampling interval of $0.1$ seconds, and by setting $m=6000$ and $n=500$.
Once the basic frequency vector $(\omega_1,\omega_2)$ is determined, the attractor can be parameterized by the time-linear coordinates $(\theta,\eta)\in [0,2\pi)^2$ with
\begin{align}
	\theta&=\omega_1 t, \nonumber\\
	\eta &= \omega_2 t. \label{eq:lineartorus}
\end{align}
The trajectory of the system on the parameterized torus is shown in \cref{fig:Torus}.
The Koopman eigenfunctions associated with the frequency $\mathbf{k}\cdot(\omega_1,\omega_2)$ are given by
\begin{align}
	\phi_\mathbf{k}=e^{i\mathbf{k}\cdot (\theta,\eta)},~\quad \mathbf{k}\in\mathbb{Z}^2.
\end{align}

The modes obtained by Hankel-DMD method provide an approximation of the Koopman eigenfunction along the trajectory (the dots in \cref{fig:Torus}) which could be extended to the whole torus through an interpolation process. Using this technique, we have plotted the Koopman eigenfunctions on the parameterized torus in  \cref{fig:Linear16k}. The computed value of frequencies and eigenfunctions are in good agreement with the results in \cite{Hassan2016} (\Cref{table:qperiodic}).

\begin{figure}
	\centerline{\includegraphics[width=1\textwidth]{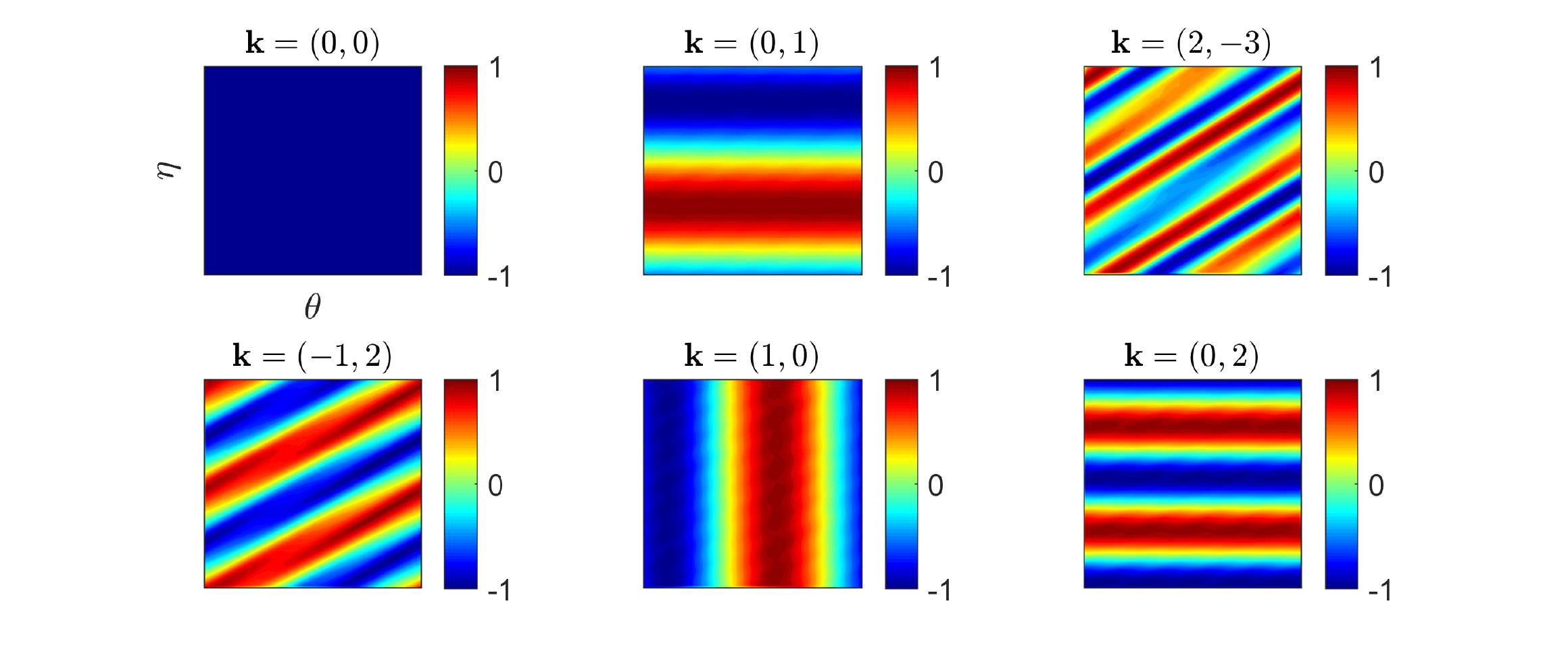}}
	\caption{Real part of Koopman eigenfunctions $\tilde \phi_{\mathbf{k}}$ on the parameterized torus of quasi-periodic cavity flow at $Re=16000$, computed using kinetic energy observable.}
	\label{fig:Linear16k}
\end{figure}

\begin{table*}[ht]
	\centering 
	\begin{tabular}{>{$}c<{$} c c c c} 
		\mathbf{k} & $\tilde \omega_\mathbf{k}$ & $\omega_\mathbf{k}$ (\cite{Hassan2016}) & relative error  & $var(\tilde \phi_\mathbf{k} -\phi_\mathbf{k})$ \\ [0.5ex]
		\hline 
		(0,0) & 0 & 0 & 0 & $1.64e-9$\\ 
		(0,1) & 0.60891 & 0.60890 & $1.58e{-5}$ & $9.58e{-4}$ \\
		(2,-3) & 0.12443 & 0.12598 &$ 1.23e{-2}$ & $1.69e{-1}$\\
		(-1,2) & 0.24159 & 0.24149  & $3.96e{-4}$ & $1.88e{-2}$\\
		(1,0) & 0.97624 & 0.97624 & $3.68e{-6}$ & $1.12e{-3}$\\
		(0,2) & 1.21781 & 1.21773 & $6.40e{-5}$ & $6.29e{-3}$ \\ [1ex] 
		\hline 
	\end{tabular}
	\caption{The dominant Koopman frequencies and eigenfunctions for quasi-periodic cavity flow computed using observations on kinetic energy.} 
	\label{table:qperiodic} 
\end{table*}

\subsection{Computation of asymptotic phase for Van der Pol oscillator}\label{sec:isochron}
We show an application of \cref{cor:EDMDbasin} by computing the asymptotic phase for trajectories of the Van der Pol oscillator. For definition of this problem, we closely follow the discussion in \cite{mauroy2012use}.
Consider the classical Van der Pol model
\begin{align}
	\dot{z}_1=z_2,~ \dot{z}_2=\mu (1-z_1^2)z_2-z_1.\quad \z:=(z_1,z_2)\in\mathbb{R}^2~\label{eq:VDP}
\end{align}
For the parameter value $\mu=0.3$, all the trajectories in the state space $\mathbb{R}^2$ converge to a limit cycle $\Gamma$ with the basic frequency $\omega_0\approx 0.995$.
However, the trajectories converge to different orbits on the limit cycle and their asymptotic phase depends on the initial condition.
The problem of determining the asymptotic phase associated with each initial condition in the state space is of great importance, e.g., in analysis and control of oscillator networks that arise in biology (see e.g. \cite{taylor2008sensitivity,danzl2009event}). The Koopman eigenfunctions provide a natural answer for this problem; if $\phi_0$ is the Koopman eigenfunction associated with $\omega_0$, the initial conditions lying on the same level set of $\phi_0$ converge to the same orbit and will have the same asymptotic phase \cite{mauroy2012use}. The methodology developed in \cite{mauroy2012use} is to compute the Koopman eigenfunction by taking the Fourier (or harmonic) average of a typical observable which requires prior knowledge of $\omega_0$.

We use a slight variation of Hankel-DMD algorithm to compute the basic (Koopman) frequency of the limit cycle and the corresponding eigenfunction $\phi_0$ in the same computation. Consider two trajectories of \cref{eq:VDP} starting at initial conditions $\z^1=(4,4)$ and $\z^2=(0,4)$.
The observable that we use is $f(\z)=z_1+z_2$, sampled at every $0.1$ second over a time interval of $35$ seconds ($m=250$ and $n=100$).
Recall from \cref{remark:ergodicsampling} that we can use various vectors of ergodic sampling with Hankel-DMD algorithm to compute the spectrum and eigenfunctions of the Koopman operator. We populate the Hankel matrices with observation on both trajectories, such that the $l-$th column of the Hankel matrix, denoted by $H^l$, is given by
\begin{align*}
		H^l= [f(\z^1),~f(\z^2),~f\circ T(\z^1),~f\circ T(\z^2),~\ldots,~f\circ T^{m-1}(\z^1),~f\circ T^{m-1}(\z^2)]^T.
\end{align*}
Similarly,
\begin{align*}
UH^l= [f\circ T(\z^1),~f\circ T(\z^2),~f\circ T^2(\z^1),~f\circ T^2(\z^2)~\ldots,~f\circ T^{m}(\z^1),~f\circ T^{m}(\z^2)]^T.
\end{align*}
The dynamic modes obtained by applying the Exact DMD to these Hankel matrices approximate the Koopman eigenfunctions along the two trajectories in the form of
\begin{align*}
\tilde \phi_0(\z^1,\z^2):=[\phi_0(\z^1),~\phi_0(\z^2),~\phi_0\circ T(\z^1),~\phi_0\circ T(\z^2),~\ldots,~\phi_0\circ T^{m-1}(\z^1),~\phi_0\circ T^{m-1}(\z^2)]^T.
\end{align*}

\Cref{fig:VanderPol} shows the agreement between Hankel-DMD and the Koopman eigenfunction obtained from Fourier averaging with known frequency \cite{mauroy2012use}. The right panel shows $\theta=\angle \tilde\phi_0$ plotted as the color field along the trajectories which characterizes the asymptotic phase of each point on the trajectories.

\begin{figure}
	\centerline{\includegraphics[width=1.1 \textwidth]{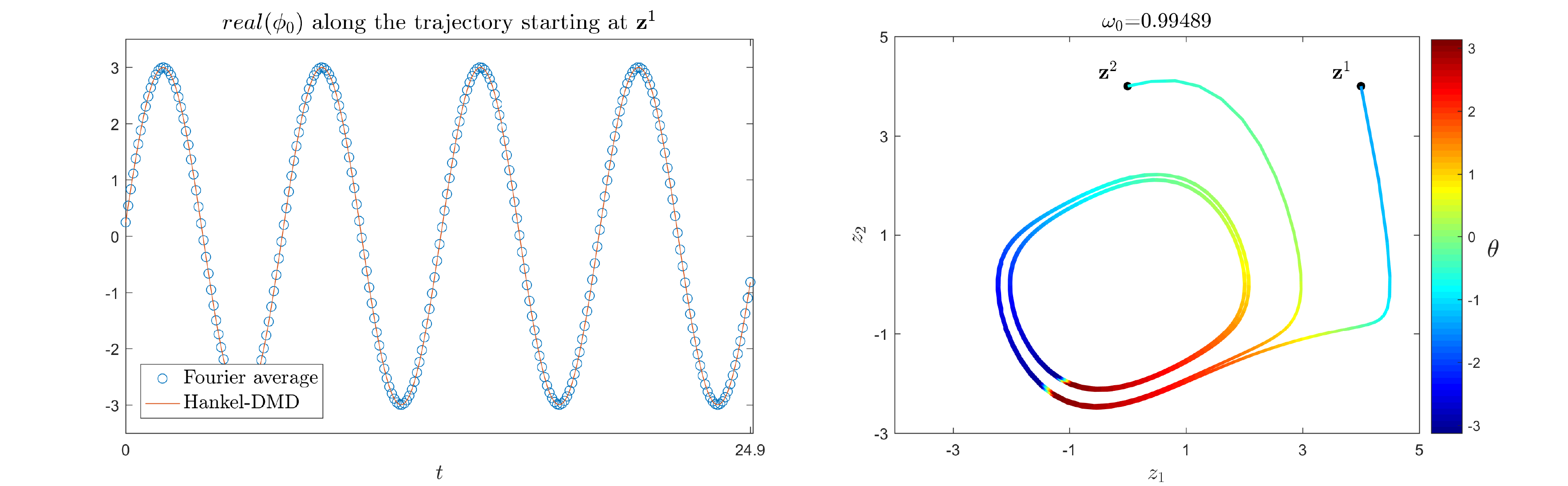}}
	\caption{Asymptotic phase for Van der Pol oscillator: The Koopman eigenfunction associated with $\omega_0$ along the trajectory starting at $\z^1$ (left), and asymptotic phase for points along two trajectories starting at $\z^1$ and $\z^2$ (right). The states with the same color converge to the same orbit on the limit cycle.}
	\label{fig:VanderPol}
\end{figure}

\subsection{Application to multiple observables: the quasi-periodic cavity flow}
We show the application of \cref{alg:HDMD} with multiple observables by revisiting the example of quasi-periodic flow in \cref{sec:Cavity}. Let $\{G_i:=G(t_0+i\Delta t)\}$ be the set of measurements of the stream function at a point on the flow domain ($x=y=0.3827$ in the domain defined in \cite{Hassan2016}). Also let $E$ be the kinetic energy of the flow.  We use the Hankel matrices of observations on $G$ and $E$, by setting $m=6000$ and $n=500$ and form the data matrices
\begin{align}
	X=\left[\begin{matrix}
	\tilde H_E & \tilde \alpha \tilde H_G
	\end{matrix}\right] ,\quad
	Y=\left[\begin{matrix}
	U\tilde H_E & \tilde \alpha U\tilde H_G
	\end{matrix}\right].
\end{align}
where we have approximated the scaling factor $\tilde \alpha$ by
\begin{align}
\tilde \alpha = \frac{\|\tilde G_{m}\|}{\|\tilde E_{m}\|}.
\end{align}
\Cref{table:qperiodic2} shows the error in approximation of Koopman frequencies and eigenvalues using two observables. The accuracy of computation is comparable to the single-observable computation in \cref{sec:Cavity}. However, since we are supplementing the observation on $E$ with measurements on $G$, we are able to capture new eigenfunctions, three of which are shown in the  bottom row of \cref{fig:Linear16k_exact}.
\begin{figure}
	\centerline{\includegraphics[width=1 \textwidth]{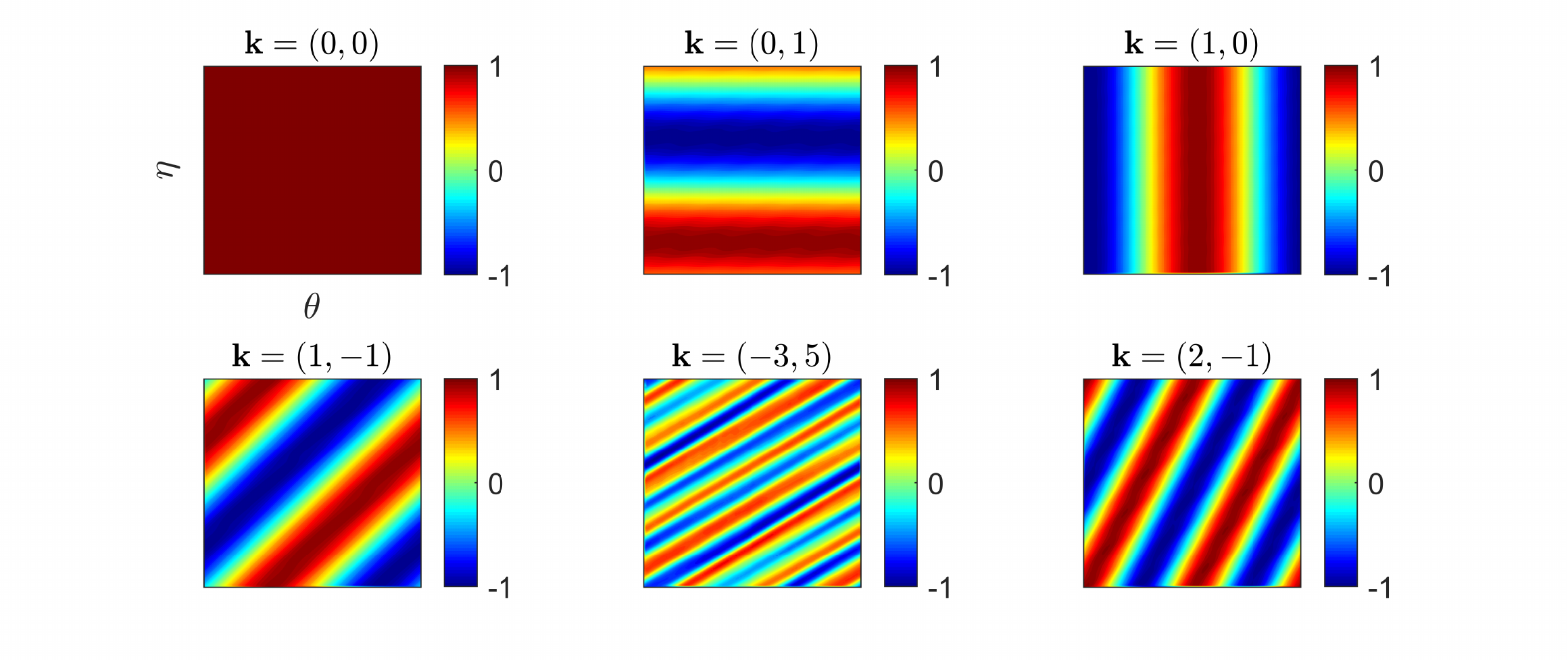}}
	\caption{Real part of Koopman eigenfunctions $\tilde \phi_{\mathbf{k}}$ on the parameterized torus of quasi-periodic cavity flow at $Re=16000$, computed using two observables: kinetic energy and stream function.}
	\label{fig:Linear16k_exact}
\end{figure}

\begin{table*}[ht]
	\centering 
	\begin{tabular}{>{$}c<{$} c c c c} 
		\mathbf{k} & $\tilde \omega_\mathbf{k}$ & $\omega_\mathbf{k}$ (\cite{Hassan2016}) & relative error  & $var(\tilde \phi_\mathbf{k} -\phi_\mathbf{k})$ \\ [0.5ex]
		\hline 
		(0,0) & 0 & 0 & 0 & $6.76e-9$\\ 
		(0,1) & 0.60892 & 0.60890 & $1.67e{-5}$ & $9.38e{-4}$ \\
		(1,0) & 0.97624 & 0.97624 &$ 5.86e{-6}$ & $5.90e{-3}$\\
		(1,-1) & 0.36732 & 0.36728  & $1.13e{-4}$ & $6.98e{-4}$\\
		(-3,5) & 0.11680 &  ---	&		---	 & $1.34e{-1}$\\
		(2,-1) & 1.34353 & 1.34352 & $3.68e{-6}$ & $1.10e{-2}$ \\ [1ex] 
		\hline 
	\end{tabular}
	\caption{Accuracy of the dominant Koopman frequencies and eigenfunctions for quasi-periodic cavity flow computed using Hankel-DMD method with two observables.}
	\label{table:qperiodic2} 
\end{table*}

\section{Summary and future work}
\label{sec:conclusion}
In this paper, we have studied the convergence of DMD algorithms for systems with ergodic attractors. Our approach is based on approximation of the function projections using the vector projection in DMD, which is made possible by the Birkhoff's ergodic theorem.
 We showed a precise connection between SVD, which is frequently used in the DMD algorithms, and POD of ensemble of observables on the state space. By applying SVD to Hankel-embedded time series, we gave a new representation of chaotic dynamics on mixing attractors based on the evolution of coordinates in the space of observables. This representation turns out to be more beneficial for analysis and control purposes than the classic state-space trajectory-based representation.
To compute the discrete Koopman operator for systems with ergodic attractors, we introduced the Hankel-DMD algorithm which is equivalent of applying the classic DMD algorithm to Hankel-type data matrices.
{ This algorithm can compute Koopman spectrum using a small number of observables and trajectories in high-dimensional systems like fluid flows. The Hankel-DMD method also shows promise for computing the dissipative eigenvalues of the Koopman operator, i.e., eigenvalues inside the unit circle. We will discuss this in future articles.}


\section*{Time series data and MATLAB codes}
The time-series data  and MATLAB codes of the numerical examples can be found at \url{https://mgroup.me.ucsb.edu/resources}.

\section*{ACKNOWLEDGMENT}
H.A. thanks Nithin Govindararjan for his comments on the initial manuscript.

\bibliographystyle{plain}
\bibliography{MOTDyS,KMDbib}
\end{document}